\documentclass[11pt]{amsart}

\usepackage[
    paper=a4paper, 
    margin=2.54cm,
]{geometry}
\usepackage{setspace}
\usepackage{url}
\usepackage[dvipsnames]{xcolor}
\usepackage{tikz}
\usepackage{tikz-cd}
\tikzcdset{
    every label/.append style = {font = \normalsize},
    ampersand replacement=\&,
}
\usepackage{amsthm}
\usepackage{amsmath}
\usepackage{amssymb}
\usepackage{mathtools}
\usepackage{mathrsfs}
\usepackage[breaklinks]{hyperref}
\hypersetup{
    colorlinks = true,
    linkcolor = OliveGreen,
    citecolor = OliveGreen,
    urlcolor = OliveGreen
}
\usepackage[capitalise]{cleveref}
\usepackage{ulem} 

\usepackage{silence}
\usepackage[
    citestyle=alphabetic,
    bibstyle=alphabetic,
    maxalphanames=99,
    maxnames=50,
    giveninits=true,
    datamodel=mrnumber,
]{biblatex}
\renewbibmacro{in:}{}
\AtEveryBibitem{%
    \clearfield{month}
    \clearfield{issn}
    \clearfield{isbn}
    \clearfield{number}
    \clearfield{url}
    \ifentrytype{book}{\clearfield{pages}}{}
    \ifentrytype{book}{\clearfield{series}}{}
    \ifentrytype{misc}{}{\clearfield{note}}
}
\DeclareNameAlias{author}{given-family}
\DeclareFieldFormat*{title}{\textit{#1}}
\DeclareFieldFormat[article]{journaltitle}{#1,\space}
\DeclareFieldFormat{volume}{\textbf{#1}\space}
\DeclareFieldFormat[incollection]{booktitle}{``#1''}
\DeclareFieldFormat{pages}{#1}
\DeclareFieldFormat{mrnumber}{%
  MR\addcolon\space
  \ifhyperref
    {\href{http://www.ams.org/mathscinet-getitem?mr=#1}{\nolinkurl{#1}}}
    {\nolinkurl{#1}}}

\renewbibmacro*{doi+eprint+url}{%
  \iftoggle{bbx:doi}
    {\printfield{doi}}
    {}%
  \newunit\newblock
  \printfield{mrnumber}%
  \newunit\newblock
  \iftoggle{bbx:eprint}
    {\usebibmacro{eprint}}
    {}%
  \newunit\newblock
  \iftoggle{bbx:url}
    {\usebibmacro{url+urldate}}
    {}}

\usepackage{enumitem}
\setlist[itemize]{noitemsep, nolistsep}
\setlist[enumerate]{noitemsep, nolistsep}
\setlist[enumerate,1]{
    label = (\alph*),
    ref = (\alph*)
}
\makeatletter
\def\namedlabel#1#2{\begingroup
    #2%
    \def\@currentlabel{#2}%
    \phantomsection\label{#1}\endgroup
}
\makeatother

\theoremstyle{definition}
\newtheorem{thm}{T}
\numberwithin{thm}{section}
\numberwithin{equation}{section}
\newtheorem{definition}[thm]{Definition}
\newtheorem{theorem}[thm]{Theorem}
\newtheorem{lemma}[thm]{Lemma}
\newtheorem{corollary}[thm]{Corollary}
\newtheorem{proposition}[thm]{Proposition}
\newtheorem{remark}[thm]{Remark}
\newtheorem{example}[thm]{Example}
\newtheorem{notation}[thm]{Notation}

\newcommand{\hypotheses}{Let $(Q,P)$ be a weakly quasi-lattice ordered group, and let $\Lambda$ be a $P$-graph with $\FA{\Lambda} \neq \emptyset$.~}
\newcommand{\hypothesesFA}{Let $(Q,P)$ be a weakly quasi-lattice ordered group, and let $\Lambda$ be a $P$-graph with $\FA{\Lambda} = \Lambda$.~}

\newcommand{\define}[1]{\emph{#1}}

\newcommand{\powerset}[1]{\mathcal{P}(#1)}
\newcommand{\setof}[2]{\left\{ #1 \;\middle|\; #2 \right\}}

\newcommand{\N}{\mathbb{N}}

\renewcommand{\r}{\boldsymbol{\mathrm{r}}}
\newcommand{\s}{\boldsymbol{\mathrm{s}}}
\renewcommand{\d}{\boldsymbol{\mathrm{d}}}
\newcommand{\FA}[1]{\mathrm{FA}(#1)}
\newcommand{\FAr}[1]{\mathrm{FA}_{\r}(#1)}
\newcommand{\tg}{\Lambda_0}
\newcommand{\Yeends}{\Lambda^\mathrm{Yee}}

\newcommand{\filters}[1]{\mathcal{F}(#1)}
\newcommand{\pathfromto}[3]{#1(#2,#3)}
\newcommand{\principal}[1]{F(#1)}
\newcommand{\ultrafilters}[1]{\mathcal{U}(#1)}
\newcommand{\PS}[1]{\mathcal{F}_{\mathrm{FA}}(#1)}
\newcommand{\BPS}[1]{\partial #1}

\newcommand{\shifton}[2]{#1 #2}
\newcommand{\shiftoff}[2]{#2 \cdot #1}
\newcommand{\rpa}[2]{#1 \cdot #2}
\newcommand{\rpadom}[1]{\textrm{dom}(#1)}
\newcommand{\rpacod}[1]{\textrm{im}(#1)}
\newcommand{\rpamap}[1]{T_{#1}}
\newcommand{\sdpg}[3]{G(#1,#2,#3)}
\newcommand{\sdpgcylinder}[4]{Z(#1,#2,#3,#4)}

\newcommand{\reduction}[2]{#1|_{#2}}
\newcommand{\composablepairs}[1]{{#1}^{(2)}}
\newcommand{\unitspace}[1]{#1^{(0)}}

\newcommand{\PG}[1]{G_{#1}}
\newcommand{\BPG}[1]{G_{\BPS{#1}}}

\newcommand{\cylindersymbol}{\mathcal{Z}}
\newcommand{\Fell}[3]{\cylindersymbol_{#1}\left(#2 \setminus #3\right)}
\newcommand{\Fellin}[2]{\cylindersymbol_{#1}\left(#2\right)}
\newcommand{\Fellgpd}[5]{\cylindersymbol_{#1}\left(#2 \setminus #3, #4 \setminus #5\right)}

\newcommand{\Spielbergstopology}{\tau^\mathrm{Spi}}
\newcommand{\Spielbergsgroupoid}[1]{G^\mathrm{Spi}_{#1}}

\begin{document}

\title[
    Path groupoids of nonfinitely aligned higher-rank graphs
]{
    A local treatment of finite alignment and path groupoids of nonfinitely aligned higher-rank graphs
}

\author[M. Jones]{Malcolm Jones}
\address[M. Jones]{
    Institute of Mathematics, Physics and Mechanics,
    Ljubljana 1000,
    Slovenia
}
\email{\href{mailto:malcolm.jones@imfm.si}{malcolm.jones@imfm.si}}

\begin{abstract}
    We give a local treatment of finite alignment by identifying the finitely aligned part of any (not necessarily finitely aligned) higher-rank graph. 
    We show the finitely aligned part is itself a constellation and forms a finitely aligned relative category of paths together with the original higher-rank graph. 
    We show that the elements of the finitely aligned part are precisely those whose cylinder sets are compact, which allows us to give novel definitions of locally compact path and boundary-path spaces for nonfinitely aligned higher-rank graphs. 
    We extend a semigroup action and the associated semidirect product groupoid developed by Renault and Williams to define ample Hausdorff path and boundary-path groupoids. 
    The groupoids are amenable for nonfinitely aligned $k$-graphs by a result of Renault and Williams. 
    In the finitely aligned case, the path groupoids coincide with Spielberg's groupoids, and the boundary-path groupoid has an inverse semigroup model via a result of Ortega and Pardo.
\end{abstract}
    
\maketitle

\section{Introduction}

Higher-rank graphs are a generalisation of directed graphs introduced in \cite{KP00} as $k$-graphs and further generalised in \cite{BSV13} to $P$-graphs whereby relations can be imposed on edges. 
There is an extensive literature stemming from \cite{KPRR97} that associates \textit{path spaces} and \textit{path groupoids} to $P$-graphs and even more general combinatorial data \cite{Spi20}. 
This has been extremely beneficial in the study of C*-algebras \cite{KPRR97}, Leavitt path algebras \cite{AAP05} and has implications for point-set topology \cite{PW05} and higher-dimensional Thompson groups \cite{Bri04, LSV24}.
The one-vertex higher-rank graphs form a family of monoids that has received special attention in the context of C*-algebras \cite{DPY08,DY09} as well as semigroup theory \cite{LV24}.

In the above settings, the path space needs to be locally compact, but not all higher-rank graphs have a locally compact path space (see \cite[Remark 3.7]{Yee07}). 
However, the path space of any \textit{finitely aligned} higher-rank graph is locally compact \cite[Proposition 5.7]{FMY05}. 
The finite alignment condition was first identified in \cite{RS05} (in the context of product systems of graphs). 
A monoid is finitely aligned if and only if it is right ideal Howson as in \cite{CG21}.
When studying the C*-algebras of $k$-graphs, finite alignment has been called the ``natural boundary of the subject'' \cite{FMY05}. 
More recently, there is growing interest in nonfinitely aligned examples. Spielberg associated ample (not necessarily Hausdorff) groupoids to each (not necessarily finitely aligned) left cancellative small category in \cite{Spi20}, including all higher-rank graphs. 
While a regular representation of the universal groupoid C*-algebra factors through the reduced groupoid C*-algebra, the induced representation of the reduced groupoid C*-algerba is not always faithful \cite[Proposition 11.4]{Spi20} (see also \cite{NS23}).
Thus, there is more to understand about groupoids associated to nonfinitely aligned data, and this is the focus of this work.

\textit{Path} groupoids are well-understood for finitely aligned higher-rank graphs, so we desire a path groupoid for general higher-rank graphs in order to see which results from the finitely aligned case can be generalised. 
Given a general higher-rank graph $\Lambda$, the unit space $X(\Lambda)$ of Spielberg's groupoid from \cite{Spi20} is the spectrum of a certain commutative C*-algebra associated to $\Lambda$. 
The elements of $X(\Lambda)$ identify with paths if $\Lambda$ is finitely aligned \cite[Theorem 6.8]{Spi20}, but not otherwise. 
In this sense, Spielberg's groupoids are not \textit{path} groupoids.

In this paper, we associate locally compact path spaces and ample Hausdorff path groupoids to higher-rank graphs that are not necessarily finitely aligned.
After establishing preliminaries in \cref{sec:prelims}, our first contribution is a local treatment of finite alignment in \cref{sec:finite_alignment}. 
Specifically, given a higher-rank graph $\Lambda$ and $\lambda \in \Lambda$, we define what it means for $\Lambda$ to be \textit{finitely aligned at $\lambda$} (\cref{def:FA}). 
The set $\FA{\Lambda} \coloneqq \{\lambda \in \Lambda \mid \Lambda \text{ is finitely aligned at } \lambda\}$ enjoys a number of properties related to the category structure on the higher-rank graph $\Lambda$ (\cref{lem:finite_alignment}). 
Though $\FA{\Lambda}$ can fail to be a higher-rank graph (\cref{sec:no_factorisation}), $\FA{\Lambda}$ is a constellation (\cref{prop:constellations}) in the sense of \cite{GH10}, which is a sort of ``one-sided category''. 
Also, by adjoining some units to $\FA{\Lambda}$, we attain a \textit{finitely aligned} relative category of paths (\cref{prop:finitely_aligned_relative_category_of_paths}) in the sense of \cite{Spi14}, even if $\Lambda$ is not finitely aligned.

In \cref{sec:path_space}, we recall the space $\filters{\Lambda}$ of filters in $\Lambda$ from \cite{BSV13}, which is a suitable locally compact path space when $\Lambda$ is finitely aligned, but can fail to be locally compact if $\Lambda$ is not finitely aligned \cite[Example 3.7]{Yee07}. 
Our contribution in this section is a new definition of \textit{locally compact} path and boundary-path spaces, $\PS{\Lambda}$ and $\BPS{\Lambda}$, whenever $\FA{\Lambda} \neq \emptyset$ (\cref{thm:path_spaces}). 
Our path space $\PS{\Lambda}$ is an open subspace of $\filters{\Lambda}$. 
\cref{thm:path_spaces} (among other results in this paper) depends on \cref{lem:cpt_char}, which says $\lambda \in \FA{\Lambda}$ if and only if the cylinder set of $\lambda$ in $\filters{\Lambda}$ is compact.

In \cref{sec:path_groupoid}, we construct a semigroup action on the path space like in \cite{RW17} (\cref{thm:PS_semigroup_action_locally_compact}), and we define the path and boundary-path groupoids to be the associated semidirect product groupoids from \cite{RW17}, which are ample Hausdorff groupoids (\cref{thm:path_groupoids}). 
By an amenability result from \cite{RW17}, the path and boundary-path groupoids are amenable for any $k$-graph (\cref{cor:amenable}). 

In \cref{sec:reconciliation}, we show our path and boundary-path groupoids are topologically isomorphic to Spielberg's groupoids from \cite{Spi14} in the finitely aligned case (\cref{thm:FA_groupoid_isomorphism}); consequently, our boundary-path groupoid is topologically isomorphic to Exel's tight groupoid of a certain inverse semigroup from \cite{OP20} (\cref{cor:inverse_semigroup_model}). 

\section{Preliminaries}
    \label{sec:prelims}

\subsection{Topology on the power set of a set}

Let $\Lambda$ be a countable set.
We describe a topology on the power set $\powerset{\Lambda}$ that is a special case of the pointwise topology on function spaces, where $\powerset{\Lambda}$ identifies with functions from $\Lambda$ to the discrete space $\{0, 1\}$ (see for example \cite[Definition 42.1]{Wil70}).
Consider the product topology on $\prod_{\lambda \in \Lambda} \{0,1\}$, where each $\{0,1\}$ is discrete.
The set $\prod_{\lambda \in \Lambda} \{0,1\}$ identifies with $\powerset{\Lambda}$ such that each standard basic open set in $\prod_{\lambda \in \Lambda} \{0,1\}$ identifies with a \define{cylinder}
\[
    \Fell{}{K_1}{K_2} 
    \coloneqq \setof{
        x \in \powerset{\Lambda}
    }{
        K_1 \subseteq x \subseteq \Lambda \setminus K_2
    },
\]
for some finite (possibly empty) $K_1, K_2 \subseteq \Lambda$.
For each $\mu \in \Lambda$ and finite $K \subseteq \Lambda$, let
\[
    \Fell{}{\mu}{K} \coloneqq \Fell{}{\{\mu\}}{K} 
    \text{ and } 
    \Fellin{}{\mu} \coloneqq \Fell{}{\mu}{\emptyset}.
\]
For each $X \subseteq \powerset{\Lambda}$, the collection of 
\begin{equation}
    \label{eq:subspace_cylinders}
    \Fell{X}{K_1}{K_2} \coloneqq \Fell{}{K_1}{K_2} \cap X,
\end{equation}
where $K_1, K_2 \subseteq \Lambda$ are finite is a basis for the subspace topology on $X$.
We define $\Fell{X}{\mu}{K}$ and $\Fellin{X}{\mu}$ similarly.

\begin{lemma}
    \label{lem:compact_Hausdorff_power_set}
    Let $\Lambda$ be a countable set.
    The collection of cylinders in $\powerset{\Lambda}$ is a basis for a second-countable compact Hausdorff topology on $\powerset{\Lambda}$.
    Moreover, each cylinder is compact.
\end{lemma}
\begin{proof} 
    The collection of cylinders is a basis as it identifies with a standard basis for the product topology on $\prod_{\lambda \in \Lambda} \{0,1\}$, which is a countable collection because $\Lambda$ is countable, so $\powerset{\Lambda}$ is second-countable.
    Tychonoff's theorem implies $\powerset{\Lambda}$ is compact, and products of Hausdorff spaces are Hausdorff.
    For all $\mu \in \Lambda$, $\powerset{\Lambda} \setminus \Fellin{}{\mu} = \Fell{}{\emptyset}{\{\mu\}}$, so $\Fellin{}{\mu}$ is closed.
    Thus, for \textit{finite} $K_1, K_2 \subseteq \Lambda$, $\Fell{}{K_1}{K_2}$ can be written as a finite intersection of closed sets, so $\Fell{}{K_1}{K_2}$ is closed and hence compact.
\end{proof}
    
Since $\powerset{\Lambda}$ is second-countable, closed sets and continuous functions are characterised as in \cite[Corollary 10.5(b)--(c)]{Wil70}.

\begin{lemma}\label{lem:closed_and_continuous}
    Let $\Lambda$ be a countable set.
    Suppose $X$ and $Y$ are subspaces of $\powerset{\Lambda}$.
    Then:
    \begin{enumerate}
        \item $C \subseteq X$ is closed if and only if $x_n \to x$ and $(x_n) \subseteq C$ implies $x \in C$; and
        \item $f \colon X \to Y$ is continuous if and only if $x_n \to x$ in $X$ implies $f(x_n) \to f(x)$ in $Y$.
    \end{enumerate}
\end{lemma}

Since $\powerset{\Lambda}$ is second-countable, we also get from Urysohn's metrization theorem that any subspace of $\powerset{\Lambda}$ is metrizable, which is a sufficient condition for compactness to be equivalent to \define{sequential compactness} (i.e. every sequence has a convergent subsequence) as in \cite[Theorem 28.2]{Mun00}. 
Thus:

\begin{lemma}\label{lem:sequential_compactness}
    Let $\Lambda$ be a countable set. 
    Every subspace of $\powerset{\Lambda}$ is compact if and only if it is sequentially compact.
\end{lemma}

If $(x_n)$ is a sequence in $\powerset{\Lambda}$ and $\lambda \in \Lambda$, we say $\lambda$ is \define{eventually inside (resp. outside)}\label{def:eventually_inside_outside} $(x_n)$ if there is $N \in \N$ such that $\lambda \in x_n$ (resp. $\lambda \notin x_n$) for all $n \geq N$.
    
\begin{lemma}\label{lem:convergence}
    Let $\Lambda$ be a countable set.
    A sequence $(x_n)$ converges to $x$ in $\powerset{\Lambda}$ if and only if, for all $\lambda \in \Lambda$:
    \begin{enumerate}
        \item[\namedlabel{Cin}{(in)}] if $\lambda \in x$, $\lambda$ is eventually inside $(x_n)$; and 
        \item[\namedlabel{Cout}{(out)}] if $\lambda \notin x$, $\lambda$ is eventually outside $(x_n)$.
    \end{enumerate}
\end{lemma}
\begin{proof}
    Suppose $x_n \to x \in \powerset{\Lambda}$ and fix $\lambda \in \Lambda$.
    If $\lambda \notin x$, then $x \in \Fell{}{\emptyset}{\{\lambda\}}$, so eventually $x_n \in \Fell{}{\emptyset}{\{\lambda\}}$.
    Thus, \ref{Cout} holds.
    A similar argument works for \ref{Cin}.
    
    For the other direction, consider a basic open set $\Fell{}{K_1}{K_2}$ containing $x$, where $K_1, K_2 \subseteq \Lambda$ are finite.
    Hence, $K_1 \subseteq x \subseteq \Lambda \setminus K_2$. 
    Let $N$ be the maximum of the indices given by applying \ref{Cin} to each element of $K_1$ and \ref{Cout} to each element of $K_2$.
    Then, for all $n \geq N$, $x_n \in \Fell{}{K_1}{K_2}$, so $x_n \to x$.
\end{proof}

\subsection{Higher-rank graphs}

Higher-rank graphs were introduced in \cite{KP00} as $k$-graphs and generalised to $P$-graphs in \cite{BSV13}.
All the examples in this paper are $k$-graphs. 
That said, we prove our results for $P$-graphs, which generalise both $k$-graphs and the hybrid graphs of \cite{Spi07}. 
The generalisation is worthwhile: Any purely infinite simple C*-algebra is stably isomorphic to the C*-algebra of a $P$-graph \cite{Spi07, BSV13}.

We follow \cite{BSV13, aHNSY21} for preliminaries on weakly quasi-lattice ordered groups and $P$-graphs.
Let $Q$ be a discrete group, and let $P$ be a semigroup with unit $e$ such that $P \cap P^{-1} = \{e\}$. 
For each $p, r \in P$, define $p \leq r$ if $pq = r$ for some $q \in P$. 
Following \cite{aHNSY21}, $(Q, P)$ is a \define{weakly quasi-lattice ordered group} if, whenever two elements of $P$ have a common upper bound with respect to $\leq$, there is a least common upper bound.
For a weakly quasi-lattice ordered group $(Q, P)$, a (discrete) \define{$P$-graph} consists of a countable small category $(\Lambda, \composablepairs{\Lambda}, \unitspace{\Lambda}, \r, \s)$ together with a functor $\d \colon \Lambda \to P$ satisfying the \define{unique factorisation property}: for every $\lambda \in \Lambda$ and $p, q \in P$ with $\d(\lambda) = pq$, there are unique elements $\mu, \nu \in \Lambda$ such that $\lambda = \mu\nu$, $\d(\mu) = p$ and $\d(\nu) = q$.

We state some basic facts and notations about $P$-graphs.
It follows from the unique factorisation property that $\Lambda$ is \define{left cancellative} (i.e. $\mu\nu = \mu\kappa$ implies $\nu = \kappa$), \define{right cancellative} (i.e. $\mu\nu = \kappa\nu$ implies $\mu = \kappa$), and has \define{no inverses} (i.e. $\mu\nu = \s(\nu)$ implies $\mu = \nu = \s(\nu)$). 
Hence, any $P$-graph is a \define{category of paths} (i.e. a left and right cancellative small category having no inverses) as in \cite[Example 2.2(6)]{Spi14}.
For each $p \in P$, we write $\Lambda^p \coloneqq \d^{-1}(p)$, and we have $\unitspace{\Lambda} = \Lambda^e$.
For each $\mu \in \Lambda$ and $E \subseteq \Lambda$, we write $\lambda E \coloneqq \setof{\lambda\mu}{\mu\in E \text{ and } \s(\lambda) = \r(\mu)}$.
For each $\mu, \lambda \in \Lambda$, define $\mu \preceq \lambda$ if $\mu\nu = \lambda$ for some $\nu \in \Lambda$. 
The relation $\preceq$ is the same as that of \cite[Definition 2.5]{Spi14} for categories of paths, so $\preceq$ is a partial order. 
Moreover, $\preceq$ is \define{left invariant} (i.e. $\mu\nu \preceq \mu\kappa$ implies $\nu \preceq \kappa$) because of left cancellation.

\begin{example}
    The $k$-graphs of \cite{KP00} are the $P$-graphs where $(Q, P) = (\mathbb{Z}^k, \N^k)$ for some $k \in \mathbb{Z}^+$.
\end{example}

We say $\Lambda$ is \define{finitely aligned} if, for all $\mu, \nu \in \Lambda$, there is a finite $J \subseteq \Lambda$ such that $\mu\Lambda \cap \nu\Lambda = \bigcup_{\lambda \in J} \lambda\Lambda$ as in \cite[Page 729]{Spi12}. 
Our first contribution is to give a local treatment of finite alignment in \cref{sec:finite_alignment}, during which we give examples of $k$-graphs.

\subsection{Groupoids}

We refer to \cite[Page 3--4]{Wil19} for preliminaries on groupoids.
A \define{groupoid} consists of a set $G$ and a subset $\composablepairs{G} \subseteq G \times G$ together with a map $(g,h) \mapsto gh$ from $\composablepairs{G}$ to $G$ (called \define{composition}) and a map $g \mapsto g^{-1}$ from $G$ to $G$ (called \define{inversion}) such that:
\begin{enumerate}
    \item if $(f, g),(g, h) \in \composablepairs{G}$, 
    then $(f, gh),(fg, h) \in \composablepairs{G}$ and 
    \(f(gh) = (fg)h \eqqcolon fgh\);
    \item for all $g \in G$, $(g^{-1})^{-1} = g$; and
    \item for all $g \in G$, $(g^{-1},g) \in \composablepairs{G}$ and $(g,h) \in \composablepairs{G}$ implies $g^{-1}(gh) = h$ and $(gh)h^{-1} = g$.
\end{enumerate}

The set $\unitspace{G} \coloneqq \setof{x \in G}{x = xx = x^{-1}}$ is called the \define{unit space} of $G$.
There are \define{range} and \define{source} maps from $G$ to $\unitspace{G}$ given by $\r(g) \coloneqq gg^{-1}$ and $\s(g) \coloneqq g^{-1}g \in \unitspace{G}$, for each $g \in G$, respectively.
We have $\r(G) = \s(G) = \unitspace{G}$, and $(g, h) \in \composablepairs{G}$ if and only if $\s(g) = \r(h)$.

We define topological, locally compact, \'etale and ample groupoids as in \cite{Ste10}, though they were first studied by Renault \cite{Ren80} and then Paterson \cite{Pat99}.
Let $G$ be a groupoid endowed with a (\textit{not necessarily Hausdorff}) topology.
The set $\composablepairs{G}$ is endowed with the subspace topology from the product topology on $G \times G$.
We say $G$ is \define{topological} if composition and inversion are continuous.
A \define{locally compact groupoid} is a topological groupoid $G$ such that $G$ is locally compact and the unit space $\unitspace{G}$ is both locally compact \textit{and Hausdorff} in the subspace topology.
A locally compact groupoid $G$ is \define{\'etale} if $\s \colon G \to \unitspace{G}$ is a local homeomorphism.
An \'etale groupoid $G$ is \define{ample} if $\unitspace{G}$ has a basis of compact open sets.

If $G$ is a groupoid, we call $H \subseteq G$ a \define{subgroupoid} if $H$ is a groupoid under the inherited operations.
Given $U \subseteq \unitspace{G}$, the \define{reduction} of $G$ to $U$ is $\reduction{G}{U} \coloneqq \s^{-1}(U) \cap \r^{-1}(U)$, which is a subgroupoid of $G$ under the inherited operations.
Following \cite{SSW20}, we say $U \subseteq \unitspace{G}$ is \define{invariant} if $\r(\s^{-1}(U)) \subseteq U$, in which case $\reduction{G}{U} = \s^{-1}(U)$.
If $G$ is \'etale and $U \subseteq \unitspace{G}$ is closed and invariant, then $\reduction{G}{U}$ is an \'etale closed subgroupoid of $G$ (see \cite[Page 96]{SSW20}).

\subsection{Semigroup actions}
    \label{sec:semigroup_action_prelims}

The following definitions come from \cite[\S 5]{RW17}.
Let $X$ be a set, and let $P$ be a subsemigroup of a group $Q$ with unit $e$.
Suppose $X * P$ is a subset of $X \times P$ and $T \colon X * P \to X$ is a function sending each $(x,m) \in X * P$ to $T((x, m)) \eqqcolon \rpa{x}{m} \in X$ and satisfying:
\begin{enumerate}
    \item[\namedlabel{S1}{(S1)}] for all $x \in X$, $(x,e) \in X * P$ and $\rpa{x}{e} = x$; and
    \item[\namedlabel{S2}{(S2)}] for all $(x,m,n) \in X \times P \times P$, $(x,mn) \in X * P$ if and only if $(x,m) \in X * P$ and $(\rpa{x}{m},n) \in X * P$, in which case $\rpa{(\rpa{x}{m})}{n} = \rpa{x}{(mn)}$.
\end{enumerate}
We call the triple $(X,P,T)$ a \define{semigroup action (of $P$ on $X$)}\label{def:semigroup_action}.
For each $m \in P$, we write
\[
    \rpadom{m} \coloneqq \setof{x \in X}{(x,m) \in X * P}
    \text{ and }
    \rpacod{m} \coloneqq \setof{\rpa{x}{m}}{x \in \rpadom{m}},
\]
and we define $\rpamap{m} \colon \rpadom{m} \to \rpacod{m}$ by $\rpamap{m}(x) = \rpa{x}{m}$, for each $x \in \rpadom{m}$. 
We say $(X,P,T)$ is \define{directed}\label{def:semigroup_action_directed} if, for all $m,n \in P$, if $\rpadom{m} \cap \rpadom{n} \neq \emptyset$, then there is $l \in P$ such that $m,n \leq l$ and $\rpadom{m} \cap \rpadom{n} \subseteq \rpadom{l}$, in which case $\rpadom{m} \cap \rpadom{n} = \rpadom{l}$ as per the note following \cite[Definition 5.2]{RW17}.
Recall that a \define{local homeomorphism}\label{def:local_homeomorphism} from a space $X$ to a space $Y$ is a continuous map $f \colon X \to Y$ such that every $x \in X$ has an open neighbourhood $U$ such that $f(U)$ is open in $Y$ and the restriction $f|_U \colon U \to f(U)$ of $f$ is a homeomorphism \cite[Page 68]{Wil70}.
A semigroup action $(X,P,T)$ is \define{locally compact}\label{def:semigroup_action_locally_compact} if $X$ is a locally compact Hausdorff space and, for all $m \in P$, $\rpadom{m}$ and $\rpacod{m}$ are open in $X$ and $\rpamap{m} \colon \rpadom{m} \to \rpacod{m}$ is a local homeomorphism.

\subsection{Semidirect product groupoids}
    \label{sec:semidirect_product_groupoid_prelims}

Semidirect product groupoids are a generalisation of Deaconu--Renault groupoids \cite{Dea95, Ren00}.
In \cite{RW17}, Renault and Williams associate a semidirect product groupoid $\sdpg{X}{P}{T}$ to each locally compact directed semigroup action $(X,P,T)$ as follows.
Let $\sdpg{X}{P}{T}$ be the set of $(x,q,y) \in X \times Q \times X$ for which there are $m,n \in P$ satisfying $q = mn^{-1}$, $x \in \rpadom{m}$, $y \in \rpadom{n}$ and $\rpa{x}{m} = \rpa{y}{n}$. 
Let $\composablepairs{\sdpg{X}{P}{T}}$ be the set of pairs $((x,q,y),(w,r,z)) \in \sdpg{X}{P}{T} \times \sdpg{X}{P}{T}$ with $y = w$. 
Then, $\sdpg{X}{P}{T}$ is a groupoid with composition 
\(
    (x,q,y)(y,r,z) \coloneqq (x,qr,z)
\)
and inversion 
\(
    (x,q,y)^{-1} \coloneqq (y,q^{-1},x),
\)
called the \define{semidirect product groupoid} of $(X,P,T)$.
The collection of 
\(
    \sdpgcylinder{U}{m}{n}{V} \coloneqq 
    \setof{
        (x,mn^{-1},y) \in \sdpg{X}{P}{T}
    }{
        x \in U, y \in V \text{ and } \rpa{x}{m} = \rpa{y}{n}
    },
\)
where $m, n \in P$ and $U, V \subseteq X$ are open, is a basis for a topology on $\sdpg{X}{P}{T}$ such that $\sdpg{X}{P}{T}$ is \'etale and Hausdorff, and $x \mapsto (x,e,x)$ is a homeomorphism from $X$ to $\unitspace{\sdpg{X}{P}{T}}$.
It follows from the definition of a basis that:

\begin{lemma}
    \label{lem:sdpg_basic_basis}
    If $\mathcal{B}$ is a basis for the topology on $X$, then the collection of $\sdpgcylinder{B}{m}{n}{C}$, where $m,n \in P$ and $B,C \in \mathcal{B}$, is a basis for the topology on $\sdpg{X}{P}{T}$.
\end{lemma}

\section{A local treatment of finite alignment}
    \label{sec:finite_alignment}

We give a local treatment of finite alignment. 
We find each $P$-graph $\Lambda$ contains a subset $\FA{\Lambda}$ that is a \textit{right constellation} in the sense of \cite{GH10,GS17}---a sort of ``one-sided category''. 
Moreover, $\FAr{\Lambda} \coloneqq \FA{\Lambda} \cup \r(\FA{\Lambda})$ is a subcategory of $\Lambda$ such that $(\FAr{\Lambda}, \Lambda)$ is a \textit{finitely aligned} relative category of paths in the sense of \cite{Spi14}. 
We use $\FA{\Lambda}$ in \cref{sec:path_space} to define a locally compact path space $\PS{\Lambda}$ for each $P$-graph $\Lambda$ with $\FA{\Lambda} \neq \emptyset$.

\begin{definition}\label{def:FA}
    Let $(Q,P)$ be a weakly quasi-lattice ordered group, and let $\Lambda$ be a $P$-graph.
    For each $\mu, \nu \in \Lambda$, we say $\Lambda$ is \define{finitely aligned at $(\mu, \nu)$} if there is a finite (possibly empty) subset $J$ of $\Lambda$ such that $\mu\Lambda \cap \nu\Lambda = \bigcup_{\lambda \in J}\lambda\Lambda$. 
    For each $\lambda \in \Lambda$, we say $\Lambda$ is \define{finitely aligned at $\lambda$} if $\Lambda$ is finitely aligned at $(\mu, \nu)$ for any $\mu \in \lambda\Lambda$ and $\nu \in \Lambda$. 
    Let $\FA{\Lambda}$ be the set of $\lambda \in \Lambda$ such that $\Lambda$ is finitely aligned at $\lambda$.
    We say $\Lambda$ is \define{finitely aligned} if $\FA{\Lambda} = \Lambda$, which coincides with the definition of finite alignment from \cite[Page 729]{Spi12}.
    If $\Lambda$ is not finitely aligned, we say $\Lambda$ is \define{nonfinitely aligned}, as in \cite{Spi14}.
\end{definition}

\begin{example}
    \label{ex:tg}
    We describe $k$-graphs in terms of their $1$-skeletons as per \cite{RSY03}.
    Consider the $2$-graph $\tg$ from \cite[Example 12.7]{Spi20}, which is given by the following $1$-skeleton with relations $\mu\beta_n = \lambda\alpha_n$, for each $n \in \N$.
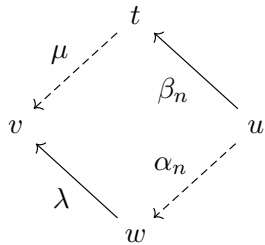
    \begin{center}
        \begin{tikzcd}[column sep=1cm, row sep=1cm]
            \& t \\
            v \&\& u \\
            \& w
            \arrow["\mu"', dashed, from=1-2, to=2-1]
            \arrow["{\beta_n}", from=2-3, to=1-2]
            \arrow["{\alpha_n}"', dashed, from=2-3, to=3-2]
            \arrow["\lambda", from=3-2, to=2-1]
        \end{tikzcd}
    \end{center}

    \noindent The only pair at which $\tg$ is not finitely aligned is $(\lambda, \mu)$.
    Thus, $\FA{\tg} = \tg \setminus \{v, \lambda, \mu\} \neq \emptyset$.
\end{example}

Our path space $\PS{\Lambda}$ (see \cref{def:PS}) is empty if and only if $\FA{\Lambda} = \emptyset$ (\cref{rem:FA_empty_iff_PS_empty}), so we restrict our attention to $\Lambda$ for which $\FA{\Lambda} \neq \emptyset$. 
Note that Spielberg's techniques for the nonfinitely aligned setting \cite{Spi20} apply even when $\FA{\Lambda} = \emptyset$.
We recall the $2$-graph $\Yeends$ from \cite{Yee07}, which satisfies $\FA{\Yeends} \neq \emptyset$ and motivates our definition of path space in \cref{sec:path_space}.

\begin{example}
    \label{ex:Yee}
    Consider the $2$-graph $\Yeends$ from \cite[Remark 3.7]{Yee07}, which is given by the following $1$-skeleton with relations $\mu_i\beta_{i,j} = \lambda\alpha_{i,j}$, for each $i,j \in \N$.

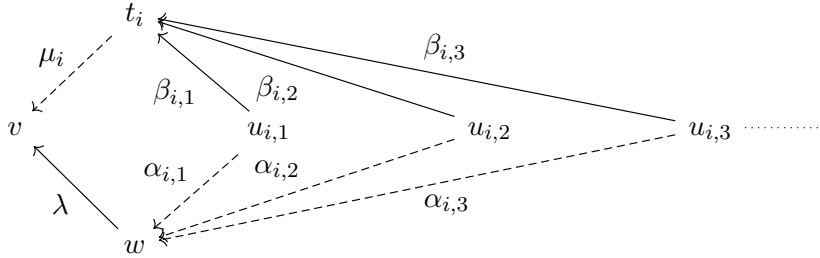
    \begin{center}
        \begin{tikzcd}[column sep=1cm, row sep=1cm]
            \& {t_i} \\
            v \&\& {u_{i,1}} \&\& {u_{i,2}} \&\& {u_{i,3}} \& {{}} \\
            \& w
            \arrow["{{\mu_i}}"', dashed, from=1-2, to=2-1]
            \arrow["{{\beta_{i,1}}}", from=2-3, to=1-2]
            \arrow["{{\alpha_{i,1}}}"', dashed, from=2-3, to=3-2]
            \arrow["{{\beta_{i,2}}}", from=2-5, to=1-2]
            \arrow["{{\alpha_{i,2}}}"', dashed, from=2-5, to=3-2]
            \arrow["{{\beta_{i,3}}}"', from=2-7, to=1-2]
            \arrow[dotted, no head, from=2-7, to=2-8]
            \arrow["{{\alpha_{i,3}}}", dashed, from=2-7, to=3-2]
            \arrow["\lambda", from=3-2, to=2-1]
        \end{tikzcd}
    \end{center}

    \noindent The importance of $\Yeends$ is that the usual path space from the finitely aligned setting is not locally compact \cite[Remark 3.7]{Yee07}, and hence does not give rise to the desired groupoids. 
    In \cref{sec:path_space}, we give a novel definition of a \textit{locally compact} path space for any $P$-graph $\Lambda$ with $\FA{\Lambda} \neq \emptyset$.
\end{example}

\subsection{Properties of the finitely aligned part of a higher-rank graph}

We collect some basic properties of $\FA{\Lambda}$.

\begin{lemma}
    \label{lem:finite_alignment}
    \hypotheses
    \begin{enumerate}
        \item\label{it:finite_alignment_right_ideal} The set $\FA{\Lambda}$ is a right ideal in $\Lambda$.
        \item\label{it:finite_alignment_closed_under_final_segments} If $\mu\delta \in \FA{\Lambda}$, then $\delta \in \FA{\Lambda}$.
        \item\label{it:finite_alignment_closed_under_source} The set $\FA{\Lambda}$ is closed under the source map $\s$.
        \item\label{it:finite_alignment_not_closed_under_range} The set $\FA{\Lambda}$ is not necessarily closed under the range map $\r$.
        \item\label{it:finite_alignment_propagates} If $\mu \in \FA{\Lambda}$, then $\s(\mu)\Lambda \subseteq \FA{\Lambda}$.
        \item\label{it:finitely_aligned} If $\mu, \nu \in \FA{\Lambda}$, then there is a finite $J \subseteq \FA{\Lambda}$ such that $\mu\Lambda \cap \nu\Lambda = \bigcup_{\lambda \in J}\lambda\Lambda$.
    \end{enumerate}
\end{lemma}

\begin{proof}
    \ref{it:finite_alignment_right_ideal} 
    \quad Suppose $\lambda \in \FA{\Lambda}$ and $(\lambda, \lambda') \in \composablepairs{\Lambda}$.
    We show $\lambda\lambda' \in \FA{\Lambda}$.
    If $(\lambda\lambda', \lambda'') \in \composablepairs{\Lambda}$ and $\nu \in \Lambda$, then $\Lambda$ is finitely aligned at $(\lambda\lambda'\lambda'', \nu)$ because $\lambda\lambda'\lambda'' \in \lambda\Lambda$ and $\Lambda$ is finitely aligned at $\lambda$. 

    \ref{it:finite_alignment_closed_under_final_segments} 
    \quad Suppose $\mu' \in \delta\Lambda$ and $\nu \in \Lambda$.
    We show $\Lambda$ is finitely aligned at $(\mu', \nu)$.
    If $\r(\nu) \neq \r(\mu')$, then $\mu'\Lambda \cap \nu\Lambda = \emptyset$, so we are done.
    Otherwise, $\r(\nu) = \r(\mu') = \r(\delta) = \s(\mu)$, so $(\mu, \nu) \in \composablepairs{\Lambda}$.
    Because $\mu\mu' \in \mu\delta\Lambda$, $\Lambda$ is finitely aligned at $(\mu\mu', \mu\nu)$.
    That is, there is a finite $J \subseteq \Lambda$ such that $\mu\mu'\Lambda \cap \mu\nu\Lambda = \bigcup_{\lambda \in J}\lambda\Lambda$.
    It follows $J \subseteq \mu\Lambda$, and so $\mu'\Lambda \cap \nu\Lambda = \bigcup_{\kappa \text{ such that }\mu\kappa \in J} \kappa\Lambda$, so $\Lambda$ is finitely aligned at $(\mu', \nu)$.

    \ref{it:finite_alignment_closed_under_source} 
    \quad This is a special case of \ref{it:finite_alignment_closed_under_final_segments} because $\lambda = \lambda \s(\lambda)$, for all $\lambda \in \Lambda$.

    \ref{it:finite_alignment_not_closed_under_range} 
    \quad The $2$-graph $\tg$ from \cref{ex:tg} is a counter-example because $v = \r(\lambda\alpha_1) \notin \FA{\tg}$.

    \ref{it:finite_alignment_propagates} 
    \quad This follows from \ref{it:finite_alignment_right_ideal} and \ref{it:finite_alignment_closed_under_final_segments}.

    \ref{it:finitely_aligned} 
    \quad Since $\Lambda$ is finitely aligned at $(\mu, \nu)$, there is a finite $J \subseteq \Lambda$ such that $\mu\Lambda \cap \nu\Lambda = \bigcup_{\lambda \in J}\lambda\Lambda$. 
    Each $\lambda \in J$ is in $\mu\Lambda$, so \ref{it:finite_alignment_right_ideal} implies $\lambda \in \FA{\Lambda}$.
\end{proof}

\subsection{The finitely aligned part as a constellation}

We highlight from \cref{lem:finite_alignment} that $\FA{\Lambda}$ is closed under composition (by \cref{lem:finite_alignment}\ref{it:finite_alignment_right_ideal}), closed under $\s$ (by \cref{lem:finite_alignment}\ref{it:finite_alignment_closed_under_source}), and not necessarily closed under $\r$ (by \cref{lem:finite_alignment}\ref{it:finite_alignment_not_closed_under_range}).
Thus, $\FA{\Lambda}$ is a right constellation in the sense of \cite[Page 280]{GS17}---a sort of ``one-sided category''. Constellations were introduced in \cite{GH10}.

\begin{proposition}
    \label{prop:constellations}
    \hypotheses
    Then, $\FA{\Lambda}$ is a right constellation with respect to the restrictions of the composition and source maps in $\Lambda$.
\end{proposition}

\subsection{Lack of a factorisation property}
    \label{sec:no_factorisation}

If $\FA{\Lambda}$ is a finitely aligned $P$-graph, we could simply apply techniques from the extensive literature on \textit{finitely aligned} $P$-graphs to $\FA{\Lambda}$. 
The immediate issue is that $\FA{\Lambda}$ can fail to be closed under $\r$ (\cref{lem:finite_alignment}\ref{it:finite_alignment_not_closed_under_range}). 
However, even the union $\FAr{\Lambda} \coloneqq \FA{\Lambda} \cup \r(\FA{\Lambda})$ may not be a $P$-graph: Recall the $2$-graph $\tg$ from \cref{ex:tg}.
Observe $\lambda\alpha_1 \in \FA{\tg}$ but $\lambda \notin \FA{\tg}$, which means $\FAr{\Lambda}$ does not have a factorisation property. 
That is, $\FAr{\tg}$ is not a $2$-graph.

\subsection{Finitely aligned relative category of paths}
    \label{sec:FA_relative_cop}

Following \cite{Spi14}, a \define{relative category of paths} is a pair $(\Lambda', \Lambda)$, where $\Lambda$ is a category of paths and $\Lambda'$ is a subcategory. 
A relative category of paths $(\Lambda', \Lambda)$ is \define{finitely aligned} if:
\begin{enumerate}
    \item for every $\mu, \nu \in \Lambda'$, there is a finite $J \subseteq \Lambda'$ such that $\mu\Lambda \cap \nu\Lambda = \bigcup_{\lambda \in J}\lambda\Lambda$; and 
    \item for every $\mu \in \Lambda'$, $\mu\Lambda \cap \Lambda' = \mu\Lambda'$.
\end{enumerate}

\begin{proposition}
    \label{prop:finitely_aligned_relative_category_of_paths}
    \hypotheses
    The pair $(\FAr{\Lambda}, \Lambda)$ is a finitely aligned relative category of paths, where 
    \[
        \FAr{\Lambda} \coloneqq \FA{\Lambda} \cup \r(\FA{\Lambda}).
    \]
\end{proposition}
\begin{proof}
    By \cref{lem:finite_alignment}, $\FA{\Lambda}$ is closed under the composition and source maps in $\Lambda$, so $\FAr{\Lambda}$ is closed under the composition, source \textit{and} range maps in $\Lambda$.
    That is, $\FAr{\Lambda}$ is a subcategory of $\Lambda$, and $\Lambda$ is a category of paths.
    Let $\mu, \nu \in \FAr{\Lambda}$.
    If $\mu$ or $\nu$ is in $\r(\FA{\Lambda})$, then $\mu\Lambda \cap \nu\Lambda = \emptyset$, $\mu\Lambda \cap \nu\Lambda = \nu\Lambda$ or $\mu\Lambda \cap \nu\Lambda = \mu\Lambda$.
    Otherwise, $\mu, \nu \in \FA{\Lambda}$, so \cref{lem:finite_alignment}\ref{it:finitely_aligned} implies there is a finite $J \subseteq \FA{\Lambda}$ such that $\mu\Lambda \cap \nu\Lambda = \bigcup_{\lambda \in J}\lambda\Lambda$. 
    It remains to show $\mu\Lambda \cap \FAr{\Lambda} = \mu\FAr{\Lambda}$. 
    Suppose $\mu\nu \in \mu\Lambda \cap \FAr{\Lambda}$. 
    Then, $\nu \in \s(\mu)\Lambda \subseteq \FA{\Lambda}$ by \cref{lem:finite_alignment}\ref{it:finite_alignment_propagates}, so $\mu\nu \in \mu\FAr{\Lambda}$. 
    Conversely, if $\mu\nu \in \mu\FAr{\Lambda}$, either $\mu\nu = \mu \in \FA{\Lambda}$, or both $\mu, \nu \in \FA{\Lambda}$ and so $\mu\nu \in \FA{\Lambda}$ because $\FA{\Lambda}$ is closed under composition (\cref{lem:finite_alignment}\ref{it:finite_alignment_right_ideal}).
\end{proof}

Spielberg associated groupoids to relative categories of paths in \cite{Spi14, Spi20}. 
We show in \cref{sec:gpd_of_relative_cop} that Spielberg's groupoid of $(\FAr{\Lambda}, \Lambda)$ does not coincide with the path groupoid $\PG{\Lambda}$ we define in \cref{sec:path_groupoid}.

\section{Path spaces}
    \label{sec:path_space}

In this section, given any $P$-graph $\Lambda$ with $\FA{\Lambda} \neq \emptyset$, we use $\FA{\Lambda}$ from \cref{sec:finite_alignment} to give a novel definition of a path space $\PS{\Lambda}$ and boundary-path space $\BPS{\Lambda}$ for $\Lambda$. 
We show in \cref{lem:cpt_char} that $\lambda \in \FA{\Lambda}$ if and only if $\Fellin{\filters{\Lambda}}{\lambda}$ is compact, where $\filters{\Lambda}$ is the set of filters of $\Lambda$, which implies $\PS{\Lambda}$ is locally compact (\cref{thm:path_spaces}).

\subsection{Filters}
    \label{sec:filters}

A \define{filter} of a $P$-graph $\Lambda$ is a nonempty subset $x$ of $\Lambda$ that is \define{hereditary} (i.e. $\lambda \preceq \mu \in x$ implies $\lambda \in x$) and \define{directed} (i.e. $\mu,\nu \in x$ implies there is $\lambda \in x$ such that $\mu,\nu \preceq \lambda$).
We denote the set of filters of $\Lambda$ by $\filters{\Lambda}$. 
Due to their defining properties, finite filters are of the form $\principal{\mu}$ for some $\mu \in \Lambda$, and infinite filters are of the form $\bigcup_{n} \principal{\mu_{n}}$ for some $\preceq$-increasing sequence $(\mu_{n}) \subseteq \Lambda$.
In this way, the filters of $\Lambda$ resemble the paths of a directed graph.
Filters have been used to give combinatorial descriptions for spaces of path-like objects for some time. 
In \cite{Nic92}, Nica showed filters provide a path model for the spectrum of the diagonal subalgebra of the Wiener--Hopf C*-algebra of a quasi-lattice ordered group. 
Paterson and Welch used filters to model paths in $k$-graphs in \cite[Lemma 2.1]{PW05}. 
Exel used filters to model paths in semigroupoids in \cite[\S 20]{Exe08}. 
Indeed, in the original paper \cite{BSV13} on $P$-graphs, filters played the role of paths, and we will do the same.

We recall basics of filters from \cite{BSV13,RW17}.
The restriction $\d|_x$ of $\d$ to any $x \in \filters{\Lambda}$ is injective \cite[Lemma 6.6(b)]{RW17}. 
For each $\lambda \in \Lambda$, we write $\principal{\lambda} \coloneqq \setof{\mu \in \Lambda}{\mu \preceq \lambda} \in \filters{\Lambda}$ \cite[Proposition 6.8(b)]{RW17}.
For any $x \in \filters{\Lambda}$, there is a unique element of $x \cap \Lambda^{(0)}$ denoted by $\r(x)$ \cite[Lemma 6.4(a)]{RW17}.
In particular, $\r(\principal{\lambda}) = \r(\lambda)$.
Like in \cite[\S 3]{BSV13}, we say $x \in \filters{\Lambda}$ is an \define{ultrafilter} of $\Lambda$ if, for all $y \in \filters{\Lambda}$, $x \subseteq y$ implies $x = y$.
We denote the set of ultrafilters of $\Lambda$ by $\ultrafilters{\Lambda}$. 
The following result is standard, but we prove it to emphasise there is no need for finite alignment.
Recall from \cref{eq:subspace_cylinders} that we define $\Fell{\filters{\Lambda}}{\mu}{K} = \Fell{}{\mu}{K} \cap \filters{\Lambda}$ for each $\mu \in \Lambda$ and finite $K \subseteq \Lambda$.

\begin{lemma}
    \label{lem:basis_for_topology_on_filters}
    \hypotheses
    The collection of 
    $\Fell{\filters{\Lambda}}{\mu}{K}$, where $\mu \in \Lambda$ and $K \subseteq \Lambda$ is finite, is a basis for $\filters{\Lambda}$ consisting of clopen sets.
\end{lemma}

\begin{remark}
    If $\Lambda$ is finitely aligned, we can assume $K \subseteq \mu\Lambda$ (see \cref{lem:FA_basis_for_topology_on_filters}).
\end{remark}

\begin{proof}[Proof of \cref{lem:basis_for_topology_on_filters}]
    Since $\filters{\Lambda}$ has the subspace topology from $\powerset{\Lambda}$, the collection of 
    \[
        \Fell{\filters{\Lambda}}{K_1}{K_2},
    \]
    where $K_1,K_2 \subseteq \Lambda$ are finite, is a basis for $\filters{\Lambda}$.
    Suppose 
    \(
        x 
        \in \Fell{\filters{\Lambda}}{K_1}{K_2} 
    \)
    for some finite $K_1,K_2 \subseteq \Lambda$.
    Because filters are nonempty, we can assume without loss of generality that $K_1$ is nonempty.
    Then, because filters are directed, there is $\mu \in x$ such that $\lambda \preceq \mu$ for all $\lambda \in K_1$.
    The hereditary property of filters implies 
    \(
        x 
        \in \Fell{\filters{\Lambda}}{\mu}{K_2} 
        \subseteq \Fell{\filters{\Lambda}}{K_1}{K_2},
    \)
    as required.
    The intersection of a closed set with a subspace is closed in the subspace topology, so each $\Fell{\filters{\Lambda}}{\mu}{K_2}$ is closed in $\filters{\Lambda}$ by \cref{lem:compact_Hausdorff_power_set}.
\end{proof}

\subsection{Path and boundary-path spaces}

If $\Lambda$ is finitely aligned, then $\filters{\Lambda}$ identifies with the unit space of the groupoid used to define the Toeplitz C*-algebra of $\Lambda$ in \cite{Spi12}; in particular, $\filters{\Lambda}$ is locally compact. 
Yeend showed in \cite[Remark 3.7]{Yee07} that $\filters{\Yeends}$ is not locally compact for the nonfinitely aligned $2$-graph $\Yeends$ from \cref{ex:Yee}. 
(Actually, Yeend uses graph morphisms instead of filters, like in the original $k$-graph paper \cite{KP00}, but the author showed in \cite{Jon24} that Yeend's arguments apply to $\filters{\Yeends}$ as well.)
We use $\FA{\Lambda}$ from \cref{sec:finite_alignment} to identify an open subspace $\PS{\Lambda}$ of $\filters{\Lambda}$ that is locally compact even if $\Lambda$ is not finitely aligned. 
The definition of $\PS{\Lambda}$ also allows a locally compact and directed semigroup action of $P$ on $\PS{\Lambda}$ in \cref{sec:path_groupoid}.

\begin{definition}[Path space]
    \label{def:PS}
    \hypotheses 
    The \define{path space} of $\Lambda$ is the set $\PS{\Lambda}$ of $x \in \filters{\Lambda}$ for which $x \cap \FA{\Lambda} \neq \emptyset$.
\end{definition}

\begin{remark}
    \label{rem:FA_empty_iff_PS_empty}
    Notice $\FA{\Lambda} \neq \emptyset$ implies $\PS{\Lambda} \neq \emptyset$ because, for each $\lambda \in \FA{\Lambda}$, $\principal{\lambda} \in \PS{\Lambda}$.
    On the other hand, if $\FA{\Lambda} = \emptyset$, then $\PS{\Lambda} = \emptyset$.
\end{remark}

\begin{example}
    \label{ex:tg_PS}
    Recall $\tg$ from \cref{ex:tg}.
    The set of filters is 
    \[
        \filters{\tg} = \{\{t\},\{u\},\{v\},\{w\},\principal{\mu},\principal{\lambda}\} \cup \{\principal{\alpha_n}, \principal{\beta_n}, \principal{\mu\beta_n} \mid n \in \N\}.
    \]
    Since $\FA{\tg} = \tg \setminus \{v, \lambda, \mu\}$, we have
    \(
        \PS{\tg} = \filters{\tg} \setminus \{\{v\}, \principal{\lambda}, \principal{\mu}\}.
    \)
\end{example}

\begin{lemma}
    \label{lem:PS_characterisation}
    \hypotheses 
    For any $x \in \filters{\Lambda}$, the following are equivalent:
    \begin{enumerate}
        \item\label{it:PS_original} $x \in \PS{\Lambda}$;
        \item\label{it:PS_stronger} $x \cap \FA{\Lambda} \cap \lambda\Lambda \neq \emptyset$ for all $\lambda \in x$.
    \end{enumerate}
\end{lemma}
\begin{proof}
    We have that \ref{it:PS_stronger} implies \ref{it:PS_original} because filters are nonempty.
    Suppose $x \in \PS{\Lambda}$, and fix $\lambda \in x$. 
    There is some $\mu \in x \cap \FA{\Lambda}$.
    Filters are directed, so there is $\nu \in x$ with $\lambda, \mu \preceq \nu$.
    In particular, $\nu \in \lambda\Lambda$, and $\nu \in \FA{\Lambda}$ because $\FA{\Lambda}$ is a right ideal in $\Lambda$ (\cref{lem:finite_alignment}\ref{it:finite_alignment_right_ideal}).
\end{proof}

\begin{definition}[Boundary-path space]
    \label{def:BPS}
    \hypotheses 
    The \define{boundary-path space} $\BPS{\Lambda}$ of $\Lambda$ is the closure of $\ultrafilters{\Lambda} \cap \PS{\Lambda}$ in $\PS{\Lambda}$.
\end{definition}

\subsection{Local compactness of the path spaces}

In the remainder of this section, we prove topological properties of $\filters{\Lambda}$ and $\BPS{\Lambda}$, which are summarised in \cref{thm:path_spaces}. 
The main work is to show they are locally compact. 
This involves the following lemma that characterises the elements of $\FA{\Lambda}$ in terms of the topology on $\filters{\Lambda}$. 
A consequence of the lemma is that the definition of $\PS{\Lambda}$ in this paper is equivalent to that of the author's thesis \cite{Jon24}.

\begin{lemma}\label{lem:cpt_char}
    \hypotheses
    For any $\lambda \in \Lambda$, the following are equivalent:
    \begin{enumerate}
        \item\label{it:cpt_char} $\lambda \in \FA{\Lambda}$;
        \item\label{it:cpt} $\Fellin{\filters{\Lambda}}{\lambda}$ is compact.
    \end{enumerate}
\end{lemma}

\begin{proof}
    Suppose $\Fellin{\filters{\Lambda}}{\lambda}$ is compact.
    We show $\Lambda$ is finitely aligned at $\FA{\Lambda}$.
    Fix $\lambda' \in \lambda\Lambda$ and $\mu \in \Lambda$.
    If $\lambda'\Lambda \cap \mu\Lambda$ is finite, then indeed $\Lambda$ is finitely aligned at $(\lambda', \mu)$.
    Otherwise, we can write $\lambda'\Lambda  \cap \mu\Lambda = (\nu_n)$ because $\Lambda$ is countable.
    Suppose for a contradiction that $\Lambda$ is not finitely aligned at $(\lambda', \mu)$.
    Then, letting $J_n \coloneqq \{\nu_1, \nu_2, \dots, \nu_n\}$, for all $n$, we have $\lambda'\Lambda \cap \mu\Lambda \neq \bigcup_{\nu \in J_n}\nu\Lambda \subseteq \lambda'\Lambda \cap \mu\Lambda$.
    Thus, there is a sequence $(\kappa_n)$ such that $\kappa_n \in (\lambda'\Lambda \cap \mu\Lambda) \setminus \bigcup_{\nu \in J_n}\nu\Lambda$ for all $n$.
    In particular, $(\principal{\kappa_n}) \subseteq \Fellin{\filters{\Lambda}}{\lambda'}$.
    Since $\Fellin{\filters{\Lambda}}{\lambda}$ is compact and $\lambda' \in \lambda\Lambda$, $\Fellin{\filters{\Lambda}}{\lambda'}$ is compact too.
    By \cref{lem:sequential_compactness}, $\Fellin{\filters{\Lambda}}{\lambda'}$ is sequentially compact, so there is a subsequence $(\principal{\kappa_{n_k}})$ of $(\principal{\kappa_n})$ and there is an $x \in \Fellin{\filters{\Lambda}}{\lambda'}$ such that $\principal{\kappa_{n_k}} \to x$.
    Because $\mu \in \principal{\kappa_n}$ for all $n$, \cref{lem:convergence} implies $\mu \in x$.
    Filters are directed, so $x \cap \lambda'\Lambda \cap \mu\Lambda \neq \emptyset$ and $\lambda'\Lambda \cap \mu\Lambda = (\nu_n)$, so there is $n'$ such that $\nu_{n'} \in x \cap \lambda'\Lambda \cap \mu\Lambda$.
    Because $\nu_{n'} \in x$ and $\principal{\kappa_{n_k}} \to x$, $\nu_{n'}$ is eventually inside $(\principal{\kappa_{n_k}})$ by \cref{lem:convergence}.
    That is, there is $K \in \N$ such that $\nu_{n'} \in \principal{\kappa_{n_k}}$ for all $k \geq K$.
    By definition of $\principal{\kappa_{n_k}}$, we have $\kappa_{n_k} \in \nu_{n'}\Lambda \subseteq \bigcup_{\nu \in J_{n'}} \nu\Lambda$ for all $k \geq K$.
    Recall, for each $n$, $\kappa_n$ was chosen so that $\kappa_n \notin \bigcup_{\nu \in J_n}\nu\Lambda$.
    Thus, we nearly have a contradiction.
    We just need to `wait' for $n_k$ to surpass $n'$.
    Formally, since $n_k$ is unbounded, there is $K'$ such that, for all $k \geq K'$, we have $n' \leq n_k$.
    Let $N \coloneqq \max\{K, K'\}$.
    Then, $N \geq K$ implies $\kappa_{n_N} \in \bigcup_{\nu \in J_{n'}} \nu\Lambda$.
    Also, $N \geq K'$ implies $n' \leq n_N$, and so $J_{n'} \subseteq J_{n_N}$.
    Therefore, $\kappa_{n_N} \in \bigcup_{\nu \in J_{n'}} \nu\Lambda \subseteq \bigcup_{\nu \in J_{n_N}} \nu\Lambda$, which is a contradiction.

    We prove \ref{it:cpt_char} implies \ref{it:cpt} by slightly modifying the proof of \cite[Lemma 5.2.1]{Jon24}.
    Suppose $\lambda \in \FA{\Lambda}$.
    If $\lambda\Lambda$ is finite, then $\Fellin{\filters{\Lambda}}{\lambda}$ is finite and therefore compact.
    Now suppose $\lambda\Lambda$ is infinite and write $\lambda\Lambda = (\mu_n)$ such that $\mu_0 = \lambda$ and $\mu_i \neq \mu_j$ whenever $i \neq j$.
    Let $(x_j)$ be a sequence in $\Fellin{\filters{\Lambda}}{\lambda}$.
    We construct a convergent subsequence of $(x_j)$ via an infinite collection of nested subsequences $(x_{f_n(j)})$, each having the property that some extension of $\mu_n$ is in $x_{f_n(j)}$, for all $j$. 
    These extensions give rise to a particular $x \in \Fellin{\filters{\Lambda}}{\lambda}$ and the sequence $(x_{f_i(i)})$ converges to $x$.
    
    Let $\lambda_0 \coloneqq \mu_0 = \lambda$, and let $(x_{f_0(j)}) = (x_j)$.
    Notice that 
    \(
        \lambda_0 \in x_{f_0(j)},
    \)
    for all $j$.
    Consider $\mu_1$.
    Suppose $\mu_1$ is contained in infinitely many $x_{f_0(j)}$.
    That is, there is a subsequence $(x_{f_1'(j)})$ of $(x_{f_0(j)})$ such that $\mu_1 \in x_{f_1'(j)}$ for all $j$.
    Since $\lambda_0 = \lambda \in \lambda\Lambda$, $\Lambda$ is finitely aligned at $(\lambda_0, \mu_1)$, so there is a finite $J_1 \subseteq \Lambda$ such that $\lambda_0\Lambda \cap \mu_1\Lambda = \bigcup_{\gamma \in J_1} \gamma\Lambda$.
    For every $j$, $\lambda_0$ and $\mu_1$ are in $x_{f_1'(j)}$, which is directed, so $\emptyset \neq \lambda_0\Lambda \cap \mu_1\Lambda \cap x_{f_1'(j)} \subseteq \bigcup_{\gamma \in J_1} \gamma\Lambda \cap x_{f_1'(j)}$.
    For all $j$, $x_{f_1'(j)}$ is hereditary, so there is a $\gamma_j \in J_1 \cap x_{f_1'(j)}$.
    Since $J_1$ is finite, the pigeonhole principle implies there is $\gamma \in J_1$ such that, for infinitely many $j$, $\gamma \in x_{f_1'(j)}$.
    Let $(x_{f_1(j)})$ be the subsequence of $(x_{f_1'(j)})$ given by those infinitely many $j$, and let $\lambda_1 \in \Lambda$ be such that $\lambda_0\lambda_1 = \gamma$.
    Thus, $\lambda_0\lambda_1 \in x_{f_1(j)}$ for all $j$.
    Observe $\mu_1 \preceq \lambda_0\lambda_1$.
    On the other hand, if $\mu_1$ is contained in only finitely many $x_{f_0(j)}$, let $(x_{f_1(j)}) \coloneqq (x_{f_0(j)})$, and let $\lambda_1 \coloneqq \s(\lambda_0)$.
    In both cases,
    \(
        \lambda_0\lambda_1 \in x_{f_1(j)},
    \)
    for all $j$.
    
    Repeating the above argument with $\lambda_0\lambda_1$ instead of $\lambda_0$ and with the index of every other variable increased by $1$ produces $\lambda_2 \in \Lambda$ and a subsequence $(x_{f_2(j)})$ of $(x_{f_1(j)})$ such that $\lambda_0\lambda_1\lambda_2 \in x_{f_2(j)}$ for all $j$.
    Moreover, if $\mu_2$ is contained in infinitely many $x_{f_1(j)}$, then $\mu_2 \preceq \lambda_0\lambda_1\lambda_2$.
    Indeed, this can be repeated again for each $n \geq 3$.
    Thus, for each $n \geq 0$, we have a subsequence $(x_{f_n(j)})$ of $(x_j)$, where $(x_{f_0(j)}) = (x_j)$, and $(x_{f_{n+1}(j)})$ is a subsequence of $(x_{f_n(j)})$, for all $n \geq 0$.
    Also, we have a subset $\setof{\lambda_n}{n \geq 0} \subseteq \Lambda$ such that 
    \(
        \lambda_0\lambda_1\dots\lambda_n \in x_{f_n(j)},
    \)
    for all $n, j \geq 0$.
    In particular, for any $i \geq 0$,
    \(
        \lambda_0\lambda_1\dots\lambda_i \in x_{f_i(i)},
    \)
    Since filters are hereditary,
    \begin{equation}\label{eq:fa_implies_compact_cylinders}
        0 \leq n \leq i
        \implies 
        \lambda_0\lambda_1\dots\lambda_n \in x_{f_i(i)}.
    \end{equation}
    Lastly, for each $n \geq 1$, if $\mu_n$ is contained in infinitely many $x_{f_{n-1}(j)}$, then $\mu_n \preceq \lambda_0\lambda_1\dots\lambda_n$.

    We claim that
    \[
        x_{f_i(i)}
        \to x 
        \coloneqq \bigcup_n \principal{\lambda_0\lambda_1\dots\lambda_n}.
    \]
    Because $(\principal{\lambda_0\lambda_1\dots\lambda_n})$ is totally ordered with respect to inclusion, $x \in \filters{\Lambda}$ (see the proof of \cite[Lemma 3.2]{BSV13}).
    Moreover, $\lambda_0 = \lambda \in x$, so $x \in \Fellin{\filters{\Lambda}}{\lambda}$.
    We show $x_{f_i(i)} \to x$ using \cref{lem:convergence}.
    Fix $\mu \in x$, so $\mu \preceq \lambda_0\lambda_1\dots\lambda_n$ for some $n$.
    Then, \cref{eq:fa_implies_compact_cylinders} says $\lambda_0\lambda_1\dots\lambda_n$ is eventually inside $(x_{f_i(i)})$ (for sufficiently large $i$).
    Since each $x_{f_i(i)}$ is hereditary, $\mu$ is eventually inside $(x_{f_i(i)})$, so \ref{Cin} holds.
    We show \ref{Cout} by contrapositive.
    Suppose $\mu \in \Lambda$ is not eventually outside $(x_{f_i(i)})$.
    That is, for all $I$, there is $i_I \geq I$ such that $\mu \in x_{f_{i_I}(i_I)}$.
    Thus, $\lambda, \mu \in x_{f_{i_I}(i_I)}$ for all $I$.
    Since $\lambda \in \FA{\Lambda}$, there is a finite $J \subseteq \Lambda$ such that $\lambda\Lambda \cap \mu\Lambda = \bigcup_{\nu \in J} \nu\Lambda$.
    Like earlier, because each $x_{f_{i_I}(i_I)}$ is directed, the pigeonhole principle implies there is a subsequence $(x_{f_{i_{I_k}(i_{I_k})}})$ and $\nu \in J$ such that $\nu \in x_{f_{i_{I_k}(i_{I_k})}} \cap \lambda\Lambda \cap \mu\Lambda$ for all $k$.
    In particular, $\nu \in \lambda\Lambda$, so $\nu = \mu_{n'}$ for some $n'$.
    Notice the truncated sequence $(x_{f_{i_{I_k}}(i_{I_k})})_{i_{I_k} \geq n'-1}$ is a subsequence of $(x_{f_{n'-1}(j)})$ and $\nu = \mu_{n'}$ is in every entry of $(x_{f_{i_{I_k}}(i_{I_k})})_{i_{I_k} \geq n'-1}$.
    Thus, according to the above construction, $\mu_{n'} \preceq \lambda_0\lambda_1\dots\lambda_{n'}$.
    Altogether we have $\mu \preceq \mu_{n'} \preceq \lambda_0\lambda_1\dots\lambda_{n'} \in x$, so $\mu \in x$.
    That is, \ref{Cout} holds, and so $x_{f_i(i)} \to x \in \Fellin{\filters{\Lambda}}{\lambda}$ by \cref{lem:convergence}.
\end{proof}

\begin{theorem}
    \label{thm:path_spaces}
    \hypotheses
    The path space $\PS{\Lambda}$ and the boundary-path space $\BPS{\Lambda}$ are Hausdorff and have countable bases consisting of compact sets. 
    Moreover, $\PS{\Lambda}$ is open in $\filters{\Lambda}$.
\end{theorem}
\begin{proof}
    The space $\powerset{\Lambda}$ is Hausdorff, and $\PS{\Lambda}$ is a subspace.
    Recall from \cref{eq:subspace_cylinders} that we define $\Fell{\PS{\Lambda}}{\mu'}{K} = \Fell{}{\mu'}{K} \cap \PS{\Lambda}$ for each $\mu' \in \Lambda$ and finite $K \subseteq \Lambda$.
    By \cref{lem:basis_for_topology_on_filters}, the collection of $\Fell{\PS{\Lambda}}{\mu'}{K}$, where $\mu' \in \Lambda$ and $K \subseteq \Lambda$ is finite, is a basis for the topology on $\PS{\Lambda}$.
    For any $x \in \Fell{\PS{\Lambda}}{\mu'}{K}$, there is $\mu \in x \cap \mu'\Lambda \cap \FA{\Lambda}$ by \cref{lem:PS_characterisation}.
    Then, $x \in \Fell{\PS{\Lambda}}{\mu}{K} \subseteq \Fell{\PS{\Lambda}}{\mu'}{K}$.
    Hence, the collection of $\Fell{\PS{\Lambda}}{\mu}{K}$, where $\mu \in \FA{\Lambda}$ and $K \subseteq \Lambda$ is finite, is a basis for the topology on $\PS{\Lambda}$.
    This basis is countable because $\Lambda$ is countable.
    For each $\mu \in \FA{\Lambda}$ and finite $K \subseteq \Lambda$, $\Fellin{\PS{\Lambda}}{\mu}$ is compact by \cref{lem:cpt_char} and $\Fell{\PS{\Lambda}}{\mu}{K}$ is a closed subset of $\Fellin{\PS{\Lambda}}{\mu}$, so $\Fell{\PS{\Lambda}}{\mu}{K}$ is compact.
    The boundary-path space $\BPS{\Lambda}$ also has the stated properties because $\BPS{\Lambda}$ is a closed subspace of $\PS{\Lambda}$.
    Also, $\PS{\Lambda}$ is open in $\filters{\Lambda}$ because $\PS{\Lambda}$ equals the union of $\Fellin{\filters{\Lambda}}{\mu}$ for which $\mu \in \FA{\Lambda}$.
\end{proof}

\section{Path groupoids}
    \label{sec:path_groupoid}

In \cite[\S 6]{RW17}, Renault and Williams construct actions of semigroups $P$ on \textit{$(\r, \d)$-proper topological} $P$-graphs. 
We only consider discrete $P$-graphs here, for which 
\[
    \text{$(\r, \d)$-proper} \iff \text{row-finite} \implies \text{finitely aligned},
\]
as per \cite[\S 5.1]{Jon24}. 
Thus, in the discrete setting, we generalise the semigroup action of \cite{RW17} to any $P$-graph $\Lambda$ with $\FA{\Lambda} \neq \emptyset$. 
Then, we define the path groupoid $\PG{\Lambda}$ to be the semidirect product groupoid associated to the action as in \cref{sec:semidirect_product_groupoid_prelims}.

\subsection{Shift maps on the path space}

By \cite[Lemma 3.4]{BSV13}, the maps 
\[
    \Fellin{\filters{\Lambda}}{\lambda} 
    \to \Fellin{\filters{\Lambda}}{\s(\lambda)},
    \quad 
    x 
    \mapsto \shiftoff{\lambda}{x} 
    \coloneqq \setof{\mu \in \Lambda}{\lambda\mu \in x}
\]
and 
\[
    \Fellin{\filters{\Lambda}}{\s(\lambda)} 
    \to \Fellin{\filters{\Lambda}}{\lambda},
    \quad 
    x 
    \mapsto \shifton{\lambda}{x} 
    \coloneqq \setof{\zeta \in \Lambda}{\zeta \preceq \lambda\mu \text{ for some } \mu \in x}
\]
are well-defined. 
We call them the \define{left-shift map} and \define{right-shift map} associated to $\lambda$, respectively.
The notation ``$\shiftoff{\lambda}{x}$'' mimics the operations defined in \cite[Lemma 6.7(a)--(b)]{RW17} and differs from \cite{BSV13}, where ``$\shiftoff{\lambda}{x}$'' is denoted by ``$\lambda^* \cdot x$''.

\begin{lemma}
    \label{lem:action_property}
    Let $(Q,P)$ be a weakly quasi-lattice ordered group, and let $\Lambda$ be a $P$-graph.
    For any $(\lambda, \mu) \in \composablepairs{\Lambda}$ and $x \in \filters{\Lambda}$:
    \begin{enumerate}
        \item\label{it:action_left} if $\lambda\mu \in x$, then $\shiftoff{\mu}{(\shiftoff{\lambda}{x})} = \shiftoff{(\lambda\mu)}{x}$; and
        \item\label{it:action_right} if $\s(\mu) = \r(x)$, then $\shifton{\lambda}{(\shifton{\mu}{x})} = \shifton{(\lambda\mu)}{x}$.
    \end{enumerate}
\end{lemma}
\begin{proof}
    \ref{it:action_left} \quad
    This follows from the associativity of composition in $\Lambda$.

    \ref{it:action_right} \quad
    If $\zeta \in \shifton{\lambda}{(\shifton{\mu}{x})}$, then there is $\eta \in \shifton{\mu}{x}$ such that $\zeta \preceq \lambda\eta$ and there is $\xi \in x$ such that $\eta \preceq \mu\xi$. 
    It follows from $\eta \preceq \mu\xi$ that $\lambda\eta \preceq \lambda\mu\xi$, and so $\zeta \preceq \lambda\mu\xi$ by the transitivity of $\preceq$. 
    Hence, $\zeta \in \shifton{(\lambda\mu)}{x}$.
    The other inclusion is similar but uses the reflexivity of $\preceq$.
\end{proof}

Though finite alignment is assumed in the statement of the following result from \cite{BSV13}, it is not used in their proof, so we state the result in greater generality here.

\begin{lemma}[Lemma 3.4 of \cite{BSV13}]
    \label{lem:filters_under_shifts}
    Let $(Q,P)$ be a weakly quasi-lattice ordered group, and let $\Lambda$ be a $P$-graph.
    If $\mu \in x \in \filters{\Lambda}$ and $(\lambda, \mu) \in \composablepairs{\Lambda}$, then:
    \begin{enumerate}
        \item \label{it:inversion} $\shifton{\mu}{(\shiftoff{\mu}{x})} = x$ and $\shiftoff{\lambda}{(\shifton{\lambda}{x})} = x$; and
        \item \label{it:ultrafilters_closed_under_shifts} if $x \in \ultrafilters{\Lambda}$, then $\shiftoff{\mu}{x}, \shifton{\lambda}{x} \in \ultrafilters{\Lambda}$.
    \end{enumerate}
\end{lemma}

Left-shift maps are continuous bijections even when the $P$-graph is not finitely aligned:

\begin{lemma}\label{lem:left_shift_cts_bijection}
    Let $(Q,P)$ be a weakly quasi-lattice ordered group, and let $\Lambda$ be a $P$-graph.
    For each $(\lambda, \lambda') \in \composablepairs{\Lambda}$, the left-shift map associated to $\lambda$ is a continuous bijection whose inverse is the right-shift map, and $\shiftoff{\lambda}{(\Fellin{\filters{\Lambda}}{\lambda\lambda'})} = \Fellin{\filters{\Lambda}}{\lambda'}$.
\end{lemma}
\begin{proof}
    For injectivity, suppose $x,x' \in \Fellin{\filters{\Lambda}}{\lambda}$ satisfy $\shiftoff{\lambda}{x} = \shiftoff{\lambda}{x'}$.
    \cref{lem:filters_under_shifts}\ref{it:inversion} implies $x = \shifton{\lambda}{(\shiftoff{\lambda}{x})} = \shifton{\lambda}{(\shiftoff{\lambda}{x'})} = x'$.
    For surjectivity, fix $y \in \Fellin{\filters{\Lambda}}{\s(\lambda)}$.
    Notice $\shifton{\lambda}{y} \in \Fellin{\filters{\Lambda}}{\lambda}$, and $\shiftoff{\lambda}{(\shifton{\lambda}{y})} = y$ by \cref{lem:filters_under_shifts}\ref{it:inversion}.
    Consequently, the right-shift map is the inverse of the left-shift map.
    Also, note that $\Fellin{\filters{\Lambda}}{\lambda\lambda'}$ is a subset of $\Fellin{\filters{\Lambda}}{\lambda}$ because filters are hereditary.
    For any $x \in \Fellin{\filters{\Lambda}}{\lambda\lambda'}$, $\lambda\lambda' \in x$, so $\lambda' \in \shiftoff{\lambda}{x}$.
    Hence, $\shiftoff{\lambda}{x} \in \Fellin{\filters{\Lambda}}{\lambda'}$.
    That is, $\shiftoff{\lambda}{(\Fellin{\filters{\Lambda}}{\lambda\lambda'})} \subseteq \Fellin{\filters{\Lambda}}{\lambda'}$.
    The surjectivity argument can be modified to show $\shiftoff{\lambda}{(\Fellin{\filters{\Lambda}}{\lambda\lambda'})} \supseteq \Fellin{\filters{\Lambda}}{\lambda'}$.
    To see the left-shift map is continuous, suppose $(x_n)$ converges to $x$ in $\Fellin{\filters{\Lambda}}{\lambda}$.
    To show $(\shiftoff{\lambda}{x_n})$ converges to $\shiftoff{\lambda}{x}$, we use \cref{lem:convergence}.
    Suppose $\mu \notin \shiftoff{\lambda}{x}$, so $\lambda\mu \notin x$.
    Since \ref{Cout} holds for $(x_n)$ and $x$, $\lambda\mu$ is eventually outside $(x_n)$, and so $\mu$ is eventually outside $(\shiftoff{\lambda}{x_n})$, so \ref{Cout} holds for $(\shiftoff{\lambda}{x_n})$ and $\shiftoff{\lambda}{x}$.
    A similar argument shows \ref{Cin} holds.
\end{proof}

However, left-shift maps can fail to be open:

\begin{example}
    \label{ex:left_shift_not_open}
    Recall $\tg$ from \cref{ex:tg}.
    We write $\principal{\kappa} = \setof{\nu \in \tg}{\nu \preceq \kappa}$ for each $\kappa \in \tg$.
    \cref{lem:convergence} can be used to show $\principal{\alpha_n} \to \{w\}$ (because $w$ is the only element of $\tg$ that is eventually inside $(\principal{\alpha_n})$) but $\shifton{\lambda}{\principal{\alpha_n}} = \{v, \mu, \lambda, \lambda\alpha_n\} \to \{v, \mu, \lambda\} \neq \{v, \lambda\} = \shifton{\lambda}{\{w\}}$.
    Hence, the right-shift map associated to $\lambda$ is not continuous, and the right-shift map is the inverse of the left-shift map by \cref{lem:left_shift_cts_bijection}, so the left-shift map is not open.
\end{example}

We now turn to the path space $\PS{\Lambda}$. 
We will assume $\FA{\Lambda} \neq \emptyset$ so that the path space is nonempty. 
We first note that the path space is closed under left-shift maps:

\begin{lemma}
    \label{lem:PS_closed_under_left_shift}
    \hypotheses
    If $\lambda \in x \in \PS{\Lambda}$, then $\shiftoff{\lambda}{x} \in \PS{\Lambda}$.
\end{lemma}
\begin{proof}
    Recall from \cref{lem:PS_characterisation} that $\lambda \in x \in \PS{\Lambda}$ implies there is some $\lambda\lambda' \in x \cap \FA{\Lambda} \cap \lambda\Lambda \neq \emptyset$.
    Then, \cref{lem:finite_alignment}\ref{it:finite_alignment_closed_under_final_segments} implies $\lambda' \in \FA{\Lambda} \cap \shiftoff{\lambda}{x}$, so 
    $\shiftoff{\lambda}{x} \in \PS{\Lambda}$.
\end{proof}

However, the path space can fail to be closed under right-shift maps:

\begin{example}
    \label{ex:PS_right_shift_not_well_defined}
    In $\tg$ from \cref{ex:tg}, $w \in \FA{\tg}$, so $\{w\} \in \PS{\tg}$. 
    However, $\shifton{\lambda}{\{w\}} = \{v, \lambda\} \subseteq \tg \setminus \FA{\tg}$, so $\shifton{\lambda}{\{w\}} \notin \PS{\tg}$.
\end{example}

\subsection{Semigroup action on the path space}

We use the above shift maps to define a locally compact and directed semigroup action $T$ of $P$ on $\PS{\Lambda}$. 
The reader can recall preliminaries on semigroup actions from \cref{sec:semigroup_action_prelims}. 
We use one additional notation.

\begin{notation}
    \label{not:path_from_to}
    Recall that we write $\Lambda^m = \d^{-1}(m)$, for each $m \in P$.
    For each $x \in \filters{\Lambda}$ and $m \in P$, if $x \cap \Lambda^m$ is nonempty, then there is precisely one element $\pathfromto{x}{e}{m}$ in $x \cap \Lambda^m$ by the injectivity $\d|_x$.
    Then, the left-shift map associated to $\pathfromto{x}{e}{m}$ maps $x$ to $\shiftoff{\pathfromto{x}{e}{m}}{x}$.
\end{notation}

\begin{definition}
    \label{def:semigroup_action_on_path_space}
    \hypotheses 
    Let
    \(
        \PS{\Lambda} * P 
        \coloneqq 
        \setof{
            (x,m) \in \PS{\Lambda} \times P
        }{
            x \cap \Lambda^m \neq \emptyset
        },
    \)
    and define $T \colon \PS{\Lambda} * P \to \PS{\Lambda}$ to map each $(x,m) \in \PS{\Lambda} * P$ to $\rpa{x}{m} \coloneqq \shiftoff{\pathfromto{x}{e}{m}}{x}$.
\end{definition}

\begin{proposition}\label{prop:PS_semigroup_action}
    \hypotheses 
    Then, $(\PS{\Lambda},P,T)$ is a semigroup action.
\end{proposition}
\begin{proof}
    The function $T$ is well-defined with the codomain $\PS{\Lambda}$ by \cref{lem:PS_closed_under_left_shift}.
    We check \ref{S1} and \ref{S2} of the definition of a semigroup action. 
    For each $x \in \PS{\Lambda}$, $\r(x) \in x \cap \unitspace{\Lambda} = \Lambda^e$. 
    Thus, $\rpa{x}{e} = \shiftoff{\r(x)}{x} = x$, so \ref{S1} holds.

    Now suppose $(x,mn) \in \PS{\Lambda} * P$.
    Since $\d(\pathfromto{x}{e}{mn}) = mn$ and $\Lambda$ has the unique factorisation property, there are unique $\mu\in\Lambda^m$ and $\nu\in\Lambda^n$ such that $\pathfromto{x}{e}{mn} = \mu\nu$.
    Since $x$ is hereditary, $\mu \in x \cap \Lambda^m$, so $(x,m) \in \PS{\Lambda} * P$.
    Notice $\rpa{x}{m} = \shiftoff{\mu}{x}$.
    Because $\mu\nu \in x$, we have $\nu \in (\rpa{x}{m}) \cap \Lambda^n$.
    Thus, $(\rpa{x}{m},n) \in \PS{\Lambda} * P$.
    That is, $(x,mn) \in \PS{\Lambda} * P$ implies both $(x,m) \in \PS{\Lambda} * P$ and $(\rpa{x}{m},n) \in \PS{\Lambda} * P$.
    For the other direction, suppose $(x,m) \in \PS{\Lambda} * P$ and $(\rpa{x}{m},n) \in \PS{\Lambda} * P$.
    Say $\mu \in x \cap \Lambda^m$ and $\nu \in (\rpa{x}{m}) \cap \Lambda^n$.
    It follows that $\mu\nu \in x \cap \Lambda^{mn}$, so $(x,mn) \in \PS{\Lambda} * P$.
    Using \cref{lem:action_property}, we can compute
    \(
        \rpa{(\rpa{x}{m})}{n} 
        = \shiftoff{\nu}{(\shiftoff{\mu}{x})} 
        = \shiftoff{(\mu\nu)}{x} 
        = \rpa{x}{(mn)},
    \)
    so \ref{S2} holds.
\end{proof}

The main result using the weak quasi-lattice order property of $(Q,P)$ is \cref{prop:PS_semigroup_action_directed}.

\begin{proposition}
    \label{prop:PS_semigroup_action_directed}
    \hypotheses 
    The semigroup action $(\PS{\Lambda},P,T)$ is directed.
\end{proposition}
\begin{proof}
    Fix $(m,n) \in P \times P$ such that there is some $x \in \rpadom{m} \cap \rpadom{n}$.
    Remember this means $(x,m),(x,n) \in \PS{\Lambda} * P$.
    Say $\mu \in x \cap \Lambda^m$ and $\nu \in x \cap \Lambda^n$.
    Because $x$ is directed, there are $\mu',\nu' \in \Lambda$ such that $\mu\mu' = \nu\nu' \in x$.
    Observe $m, n \leq \d(\mu\mu') = \d(\nu\nu')$.
    Since $(Q,P)$ is a weakly quasi-lattice ordered group, there is a least upper bound $l$ of $m$ and $n$, from which it follows that $\rpadom{m} \cap \rpadom{n} \subseteq \rpadom{l}$.
\end{proof}

We embark on a series of lemmas proving $(\PS{\Lambda},P,T)$ is locally compact (\cref{thm:PS_semigroup_action_locally_compact}). 
This is one of the main technical difficulties of the work that led to the definition of our path space $\PS{\Lambda}$.
Recall from the proof of \cref{thm:path_spaces} that the collection of $\Fell{\PS{\Lambda}}{\mu}{K}$, where $\mu \in \FA{\Lambda}$ and $K \subseteq \Lambda$ is finite, is a basis of compact sets for the topology on $\PS{\Lambda}$.

\begin{lemma}
    \label{lem:PS_semigroup_action_locally_compact_rpadom_open}
    \hypotheses
    For any $m \in P$, $\rpadom{m}$ is open in $\PS{\Lambda}$.
\end{lemma}
\begin{proof}
    We have $\rpadom{m} = \bigcup_{\mu \in \Lambda^m} \Fellin{\PS{\Lambda}}{\mu}$, which is a union of basic open sets.
\end{proof}

It would now be natural to try to prove $\rpacod{m} = \bigcup_{\mu \in \Lambda^m} \Fellin{\PS{\Lambda}}{\s(\mu)}$ so that $\rpacod{m}$ is open too.
We can prove $\rpacod{m} \subseteq \bigcup_{\mu \in \Lambda^m} \Fellin{\PS{\Lambda}}{\s(\mu)}$, but the reverse containment can fail because $\PS{\Lambda}$ is not closed under right-shift maps (\cref{ex:PS_right_shift_not_well_defined}).

\begin{lemma}
    \label{lem:PS_semigroup_action_locally_compact_rpacod_open}
    \hypotheses
    For any $m \in P$, $\rpacod{m}$ is open in $\PS{\Lambda}$.
\end{lemma}
\begin{proof}
    Fix $\rpa{x}{m} \in \rpacod{m}$, where $x \in \rpadom{m}$.
    Write $\mu \coloneqq \pathfromto{x}{e}{m} \in x \cap \Lambda^m$ so that $\rpa{x}{m} = \shiftoff{\mu}{x}$.
    By \cref{lem:PS_characterisation}, there is some $\mu' \in \Lambda$ with $\mu\mu' \in x \cap \FA{\Lambda}$.
    We claim $\rpa{x}{m} \in \Fellin{\PS{\Lambda}}{\mu'} \subseteq \rpacod{m}$.
    Notice $\rpa{x}{m} = \shiftoff{\mu}{x} \in \Fellin{\filters{\Lambda}}{\mu'}$.
    Also, $\mu\mu' \in \FA{\Lambda}$ implies $\mu' \in \FA{\Lambda}$ by \cref{lem:finite_alignment}\ref{it:finite_alignment_closed_under_final_segments}, so $\Fellin{\filters{\Lambda}}{\mu'} = \Fellin{\PS{\Lambda}}{\mu'}$, which is open in $\PS{\Lambda}$.
    Now fix $y \in \Fellin{\PS{\Lambda}}{\mu'}$.
    Then, $\shifton{\mu}{y} \in \Fellin{\filters{\Lambda}}{\mu\mu'} = \Fellin{\PS{\Lambda}}{\mu\mu'} \subseteq \rpadom{m}$, and so $y = \shiftoff{\mu}{(\shifton{\mu}{y})} = \rpa{(\shifton{\mu}{y})}{m} \in \rpacod{m}$.
\end{proof}

Recall from \cref{ex:left_shift_not_open} that right-shift maps are not always continuous. 
Thus, in order to prove $T$ acts by local homeomorphisms for \cref{thm:PS_semigroup_action_locally_compact} we need the following lemma, which says that right-shift maps are \textit{continuous enough}.

\begin{lemma}
    \label{lem:right_shift_continuous_almost}
    \hypotheses
    If $(\mu, \mu') \in \composablepairs{\Lambda}$,  $\mu\mu' \in \FA{\Lambda}$ and $w_n \to w$ in $\Fellin{\filters{\Lambda}}{\mu'}$, then there is a subsequence $(w_{n_k})$ such that $\shifton{\mu}{w_{n_k}} \to \shifton{\mu}{w}$ in $\Fellin{\filters{\Lambda}}{\mu\mu'}$.
\end{lemma}
\begin{proof}
    Since $(w_n) \subseteq \Fellin{\filters{\Lambda}}{\mu'}$, we have $(\shifton{\mu}{w_n}) \subseteq \Fellin{\filters{\Lambda}}{\mu\mu'}$ by definition of the right-shift map associated to $\mu$.
    Because $\mu\mu' \in \FA{\Lambda}$, $\Fellin{\filters{\Lambda}}{\mu\mu'}$ is compact by \cref{lem:cpt_char}.
    Thus, there is a subsequence $(\shifton{\mu}{w_{n_k}})$ that converges to some $z \in \Fellin{\filters{\Lambda}}{\mu\mu'}$.
    The left-shift map associated to $\mu$ is continuous (\cref{lem:left_shift_cts_bijection}), so
    \(
        w_{n_k} 
        = \shiftoff{\mu}{(\shifton{\mu}{w_{n_k}})} 
        \to \shiftoff{\mu}{z}.
    \)
    By assumption, $w_{n_k} \to w$, so $\shiftoff{\mu}{z} = w$.
    Therefore, $z = \shifton{\mu}{w}$.
\end{proof}

\begin{theorem}
    \label{thm:PS_semigroup_action_locally_compact}
    \hypotheses 
    The semigroup action $(\PS{\Lambda},P,T)$ is locally compact and directed.
\end{theorem}
\begin{proof}
    Recall from \cref{prop:PS_semigroup_action_directed} that $(\PS{\Lambda},P,T)$ is directed.
    By \cref{thm:path_spaces}, $\PS{\Lambda}$ is a locally compact Hausdorff space.
    Now fix $m \in P$, and consider the map $\rpamap{m}$ from $\rpadom{m}$ to $\rpacod{m}$.
    We showed $\rpadom{m}$ and $\rpacod{m}$ are open in \cref{lem:PS_semigroup_action_locally_compact_rpadom_open,lem:PS_semigroup_action_locally_compact_rpacod_open}, respectively.    
    It remains to show $\rpamap{m}$ is a local homeomorphism. 
    Fix $x \in \rpadom{m}$, and say $\mu \in x \cap \Lambda^m$.
    By \cref{lem:PS_characterisation}, there is some $\mu' \in \Lambda$ such that $\mu\mu' \in x \cap \FA{\Lambda}$, which means $\mu' \in \FA{\Lambda}$ by \cref{lem:finite_alignment}\ref{it:finite_alignment_closed_under_final_segments}.
    Since $\mu\mu', \mu' \in \FA{\Lambda}$, we have $\Fellin{\filters{\Lambda}}{\mu\mu'} = \Fellin{\PS{\Lambda}}{\mu\mu'}$ and $\Fellin{\filters{\Lambda}}{\mu'} = \Fellin{\PS{\Lambda}}{\mu'}$, which are both compact by \cref{lem:cpt_char}.
    Because $\mu \in y$ for all $y \in \Fellin{\PS{\Lambda}}{\mu\mu'}$, the restriction of $\rpamap{m}$ to $\Fellin{\PS{\Lambda}}{\mu\mu'}$ coincides with the left-shift map associated to $\mu$.
    Then, \cref{lem:left_shift_cts_bijection} implies the image of $\Fellin{\PS{\Lambda}}{\mu\mu'}$ under $\rpamap{m}$ is $\Fellin{\PS{\Lambda}}{\mu'}$.
    Since the left-shift map associated to $\mu$ is a continuous bijection (\cref{lem:left_shift_cts_bijection}), it remains to show the map is open.
    Let $O$ be an open set in $\Fellin{\PS{\Lambda}}{\mu\mu'}$.
    We show $\Fellin{\PS{\Lambda}}{\mu'} \setminus \rpamap{m}(O)$ is closed.
    Suppose $(w_n) \subseteq \Fellin{\PS{\Lambda}}{\mu'} \setminus \rpamap{m}(O)$ converges to some $w \in \Fellin{\PS{\Lambda}}{\mu'}$.
    By \cref{lem:right_shift_continuous_almost}, there is a subsequence $(w_{n_k})$ such that $\shifton{\mu}{w_{n_k}} \to \shifton{\mu}{w}$ in $\Fellin{\PS{\Lambda}}{\mu\mu'}$.
    Then, $\rpamap{m}(\shifton{\mu}{w}) = \rpa{(\shifton{\mu}{w})}{m} = \shiftoff{\mu}{(\shifton{\mu}{w})} = w$. 
    Suppose for a contradiction that $w \in \rpamap{m}(O)$. 
    Then, the previous calculation implies $\shifton{\mu}{w} \in O$ because $\rpamap{m}$ is a bijection.
    Thus, $\shifton{\mu}{w_{n_k}}$ is eventually in $O$, so $w_{n_k} = \rpamap{m}(\shifton{\mu}{w_{n_k}})$ is eventually in $\rpamap{m}(O)$, which is a contradiction.
\end{proof}

\subsection{Path and boundary-path groupoids}

Recall from \cref{sec:semidirect_product_groupoid_prelims} the semidirect product groupoid associated to a semigroup action as in \cite{RW17}. 
Having constructed a locally compact and directed semigroup action $(\PS{\Lambda}, P, T)$ in \cref{thm:PS_semigroup_action_locally_compact}, we can define an \'etale and Hausdorff groupoid as follows.

\begin{definition}[Path groupoid]
    \hypotheses
    The \define{path groupoid} $\PG{\Lambda}$ of $\Lambda$ is the semidirect product groupoid $\sdpg{\PS{\Lambda}}{P}{T}$ of the semigroup action $(\PS{\Lambda},P,T)$ from \cref{thm:PS_semigroup_action_locally_compact}.
\end{definition}

\begin{lemma}
    \label{lem:BPS_closed_invariant}
    \hypotheses
    The boundary-path space $\BPS{\Lambda}$ identifies with a closed invariant subset of $\unitspace{\PG{\Lambda}}$.
\end{lemma}
\begin{proof}
    We have that $\BPS{\Lambda}$ is closed in $\PS{\Lambda}$ by \cref{def:BPS}.
    The homeomorphism $z \mapsto (z,e,z)$ from $\PS{\Lambda}$ to $\unitspace{\PG{\Lambda}}$ identifies the subset $\BPS{\Lambda} \subseteq \PS{\Lambda}$ with a subset of $\unitspace{\PG{\Lambda}}$ that is invariant if, whenever $m, n \in P$, $x \in \rpadom{m}$, $y \in \rpadom{n} \cap \BPS{\Lambda}$ satisfy $\rpa{x}{m} = \rpa{y}{n}$, we have $x \in \BPS{\Lambda}$.
    Fix $m, n, x, y$ as above.
    Say $\mu \in x \cap \Lambda^m$ and $\nu \in y \cap \Lambda^n$.
    By \cref{lem:PS_characterisation}, there is some $\mu' \in \Lambda$ such that $\mu\mu' \in x \cap \FA{\Lambda}$. 
    Then, $\Fellin{\filters{\Lambda}}{\mu\mu'}$ is compact by \cref{lem:cpt_char}.
    By definition, $\BPS{\Lambda}$ is the closure of $\ultrafilters{\Lambda} \cap \PS{\Lambda}$ in $\PS{\Lambda}$, so there is a sequence $(y_n) \subseteq \ultrafilters{\Lambda} \cap \PS{\Lambda}$ such that $y_n \to y$ in $\PS{\Lambda}$.
    Because $\PS{\Lambda}$ is a subspace of $\filters{\Lambda}$, we also have that $y_n \to y$ in $\filters{\Lambda}$, so $(y_n)$ is eventually in $\Fellin{\filters{\Lambda}}{\nu}$, so we assume without loss of generality that $\nu \in y_n$ for all $n$.
    We have $(\shiftoff{\nu}{y_n}) \subseteq \ultrafilters{\Lambda} \cap \PS{\Lambda}$ because both $\ultrafilters{\Lambda}$ and $\PS{\Lambda}$ are closed under left-shift maps (\cref{lem:filters_under_shifts}\ref{it:ultrafilters_closed_under_shifts} and \cref{lem:PS_closed_under_left_shift}).
    We also know $(\shifton{\mu}{(\shiftoff{\nu}{y_n})}) \subseteq \ultrafilters{\Lambda}$ by \cref{lem:filters_under_shifts}\ref{it:ultrafilters_closed_under_shifts}.
    Note that $y_n \to y$ in $\Fellin{\filters{\Lambda}}{\nu}$.
    Also, the left-shift map associated to $\nu$ is continuous (\cref{lem:left_shift_cts_bijection}), so
    \(
        \shiftoff{\nu}{y_n} 
        \to \shiftoff{\nu}{y} 
        = \shiftoff{\mu}{x} 
        \in \Fellin{\filters{\Lambda}}{\mu'}.
    \)
    Thus, we assume without loss of generality that $\mu' \in \shiftoff{\nu}{y_n}$ for all $n$.
    Therefore, every $\shifton{\mu}{(\shiftoff{\nu}{y_n})}$ contains $\mu\mu' \in \FA{\Lambda}$, so $(\shifton{\mu}{(\shiftoff{\nu}{y_n})}) \subseteq \ultrafilters{\Lambda} \cap \PS{\Lambda}$.
    Now \cref{lem:right_shift_continuous_almost} implies there is a subsequence $(\shiftoff{\nu}{y_{n_k}})$ such that 
    \(
        \shifton{\mu}{(\shiftoff{\nu}{y_{n_k}})} 
        \to \shifton{\mu}{(\shiftoff{\nu}{y})} 
        = \shifton{\mu}{(\shiftoff{\mu}{x})}
        = x.
    \)
\end{proof}

\begin{definition}[Boundary-path groupoid]
    \label{def:boundary_path_groupoid}
    \hypotheses
    The \define{boundary-path groupoid} $\BPG{\Lambda}$ of $\Lambda$ is the reduction $\reduction{\PG{\Lambda}}{\BPS{\Lambda}}$ of $\PG{\Lambda}$ to $\BPS{\Lambda}$.
\end{definition}

\begin{theorem}
    \label{thm:path_groupoids}
    \hypotheses
    The path groupoid $\PG{\Lambda}$ and the boundary-path groupoid $\BPG{\Lambda}$ are ample, second-countable and Hausdorff.
\end{theorem}
\begin{proof}
    The path groupoid $\PG{\Lambda}$ is the semidirect product groupoid of the locally compact and directed semigroup action $(\PS{\Lambda},P,T)$ from \cref{thm:PS_semigroup_action_locally_compact}, so $\PG{\Lambda}$ is \'etale and Hausdorff. 
    Reductions of \'etale groupoids to closed invariant subsets of their unit space are closed \'etale Hausdorff subgroupoids, so $\BPG{\Lambda}$ is \'etale because of \cref{lem:BPS_closed_invariant}.
    The unit spaces $\unitspace{\PG{\Lambda}}$ and $\unitspace{\BPG{\Lambda}}$
    are homeomorphic to $\PS{\Lambda}$ and $\BPS{\Lambda}$,
    respectively, which have bases of compact sets by \cref{thm:path_spaces}.
    Hence, $\PG{\Lambda}$ and $\BPG{\Lambda}$ are ample.
    The bases from \cref{thm:path_spaces} are also countable, so 
    \cref{lem:sdpg_basic_basis} yields countable bases for
    $\PG{\Lambda}$ and the subspace $\BPG{\Lambda}$.
    Therefore, $\PG{\Lambda}$ and $\BPG{\Lambda}$ are second-countable.
\end{proof}

\subsection{Sufficient condition for amenability of the groupoids}

Our path groupoid $\PG{\Lambda}$ is second-countable, Hausdorff and \'etale, so the three predominant notions of amenability (measurewise, topological and Borel) are all equivalent (see \cite{AR00} and \cite[Corollary 2.15]{Ren15}).
Also, the boundary-path groupoid $\BPG{\Lambda}$ is a closed subgroupoid of $\PG{\Lambda}$, so $\BPG{\Lambda}$ is amenable whenever $\PG{\Lambda}$ is amenable \cite[Proposition 10.1.14]{SSW20}.
Renault and Williams supply a sufficient condition for the amenability of semidirect product groupoids in \cite[Theorem 5.13]{RW17}. 
This allows us to generalise various amenability results for path groupoids in the setting of discrete $P$-graphs, including \cite[Proposition 6.2]{Yee07}, \cite[Corollary 6.19]{RW17}, and \cite[Corollary 4.3]{RSWY18}.

\begin{proposition}[Theorem 5.13 of \cite{RW17}]
    \label{prop:amenable}
    \hypotheses
    If $Q$ is countable and amenable, then $\PG{\Lambda}$ and $\BPG{\Lambda}$ are amenable. 
\end{proposition}

\begin{corollary}
    \label{cor:amenable}
    For any $k$-graph $\Lambda$ with $\FA{\Lambda} \neq \emptyset$, $\PG{\Lambda}$ and $\BPG{\Lambda}$ are amenable. 
\end{corollary}

\subsection{Characterising the elements of the groupoids}

We characterise the elements of our path groupoids in \cref{prop:PS_groupoid_characterisation}, which is more subtle than it may seem (\cref{rem:PS_right_shift_not_well_defined}).

\newpage 
\begin{proposition}
    \label{prop:PS_groupoid_characterisation}
    \hypotheses
    For any $g \in \PS{\Lambda} \times Q \times \PS{\Lambda}$, the following are equivalent:
    \begin{enumerate}
        \item\label{PG1} $g \in \PG{\Lambda}$;
        \item\label{PG2} $g = (\shifton{\mu}{x},\d(\mu)\d(\nu)^{-1},\shifton{\nu}{x})$ for some $x \in \PS{\Lambda}$ and $\mu, \nu \in \Lambda$ with $\s(\mu) = \s(\nu) = \r(x)$.
    \end{enumerate}
\end{proposition}

\begin{remark}
    \label{rem:PS_right_shift_not_well_defined}
    The hypothesis that $g \in \PS{\Lambda} \times Q \times \PS{\Lambda}$ in the statement of \cref{prop:PS_groupoid_characterisation} is crucial. 
    Otherwise, \ref{PG2} would not imply \ref{PG1} because there are $x \in \PS{\Lambda}$ and $\mu \in \Lambda$ with $\s(\mu) =\r(x)$ such that $\shifton{\mu}{x} \notin \PS{\Lambda}$ (see \cref{ex:PS_right_shift_not_well_defined}).
\end{remark}

\begin{proof}[Proof of \cref{prop:PS_groupoid_characterisation}]
    Fix $g = (y,q,z) \in \PS{\Lambda} \times Q \times \PS{\Lambda}$.
    If $g \in \PG{\Lambda}$, then there are $m,n \in P$ such that $q = mn^{-1}$, $y \in \rpadom{m}$, $z \in \rpadom{n}$, and $\rpa{y}{m} = \rpa{z}{n}$.
    Say $\mu \in \Lambda^m \cap y$ and $\nu \in \Lambda^n \cap z$.
    Let $x \coloneqq \shiftoff{\mu}{y} = \rpa{y}{m} = \rpa{z}{n} = \shiftoff{\nu}{z} \in \PS{\Lambda}$. 
    Then, we have $g = (\shifton{\mu}{x},\d(\mu)\d(\nu)^{-1},\shifton{\nu}{x})$.
    Now suppose $g = (\shifton{\mu}{x},\d(\mu)\d(\nu)^{-1},\shifton{\nu}{x})$ for some $x \in \PS{\Lambda}$ and $\mu, \nu \in \Lambda$ with $\s(\mu) = \s(\nu) = \r(x)$.
    We know $\shifton{\mu}{x} = y \in \PS{\Lambda}$ by assumption.
    Since $\mu \in \Lambda^{\d(\mu)} \cap y$, we have $(y,\d(\mu)) \in \PS{\Lambda} * P$.
    Similarly, $(z,\d(\nu)) \in \PS{\Lambda} * P$.
    Therefore, $y \in \rpadom{\d(\mu)}$, $z \in \rpadom{\d(\nu)}$, and
    \[
        \rpa{y}{\d(\mu)}
        = \shiftoff{\mu}{y}
        = \shiftoff{\mu}{(\shifton{\mu}{x})}
        = x
        = \shiftoff{\nu}{(\shifton{\nu}{x})}
        = \shiftoff{\nu}{z}
        = \rpa{z}{\d(\nu)}.
    \]
    Hence, $g \in \sdpg{\filters{\Lambda}}{P}{T}$.
\end{proof}

\subsection{Basis for the topology on the groupoids}

Recall from \cref{lem:sdpg_basic_basis} that, if $\mathcal{B}$ is a basis for a space $X$ and $\sdpg{X}{P}{T}$ is the semidirect product groupoid of an action $T$ of $P$ on $X$, then the collection of $\sdpgcylinder{B}{m}{n}{C}$, where $m,n \in P$ and $B,C \in \mathcal{B}$, is a basis for the topology on $\sdpg{X}{P}{T}$.
Using the basis for the topology on $\PS{\Lambda}$ from \cref{thm:path_spaces}, the collection of
\(
    \sdpgcylinder{\Fell{\PS{\Lambda}}{\mu'}{J}}{m}{n}{\Fell{\PS{\Lambda}}{\nu'}{K}},
\)
where $m,n \in P$, $\mu', \nu' \in \FA{\Lambda}$ and $J,K$ are finite subsets of $\Lambda$, is a basis for the topology on $\PG{\Lambda}$.
However, we show in the following that the parameters $m$ and $n$ are redundant.

\begin{proposition}
    \label{prop:PS_groupoid_cylinders}
    \hypotheses
    The collection of 
    \[
        \Fellgpd{\PS{\Lambda}}{\mu}{J}{\nu}{K}
        \coloneqq
        \sdpgcylinder{\Fell{\PS{\Lambda}}{\mu}{J}}{\d(\mu)}{\d(\nu)}{\Fell{\PS{\Lambda}}{\nu}{K}}
    \]
    where $\mu, \nu \in \FA{\Lambda}$ and $J,K$ are finite subsets of $\Lambda$, is a basis for the topology on $\PG{\Lambda}$.    
\end{proposition}
\begin{proof}    
    Suppose $(x,mn^{-1},y)$ is in $\sdpgcylinder{\Fell{\PS{\Lambda}}{\mu'}{J}}{m}{n}{\Fell{\PS{\Lambda}}{\nu'}{K}}$ for some $m,n \in P$, $\mu', \nu' \in \FA{\Lambda}$, and finite subsets $J,K$ of $\Lambda$.
    Notice $\mu' \in x$ and $\nu' \in y$.
    Also, since $x \in \rpadom{m}$ and $y \in \rpadom{n}$, there are $\mu_m \in x \cap \Lambda^m$ and $\nu_n \in y \cap \Lambda^n$ with 
    \(
        w 
        \coloneqq \shiftoff{\mu_m}{x}
        = \rpa{x}{m}
        = \rpa{y}{n}
        = \shiftoff{\nu_n}{y}.
    \)
    Because $x$ and $y$ are directed, there are $\mu_m' \in \Lambda$ and $\nu_n' \in \Lambda$ such that $\mu', \mu_m \preceq \mu_m\mu_m' \in x$ and $\nu', \nu_n \preceq \nu_n\nu_n' \in y$.
    Then $\mu_m' \in \shiftoff{\mu_m}{x} = w$ and $\nu_n' \in \shiftoff{\nu_n}{y} = w$.
    That is, $\mu_m',\nu_n' \in w$ and $w$ is directed, so there are $\mu_m'',\nu_n'' \in \Lambda$ such that $\mu_m'\mu_m'' = \nu_n'\nu_n'' \in w$.
    Let $\mu \coloneqq \mu_m\mu_m'\mu_m''$, and let $\nu \coloneqq \nu_n\nu_n'\nu_n''$.
    Since $\mu' \preceq \mu$ and $\nu' \preceq \nu$, we know $\mu, \nu \in \FA{\Lambda}$ because $\FA{\Lambda}$ is a right ideal in $\Lambda$ (\cref{lem:finite_alignment}\ref{it:finite_alignment_right_ideal}).
    Then, $x \in \Fell{\PS{\Lambda}}{\mu}{J}$ and $y \in \Fell{\PS{\Lambda}}{\nu}{K}$.
    Also, because $\mu_m'\mu_m'' = \nu_n'\nu_n''$ we have 
    \[
        \d(\mu)\d(\nu)^{-1}
        = \d(\mu_m)\d(\mu_m'\mu_m'')\d(\nu_n'\nu_n'')^{-1}\d(\nu_n)^{-1}
        = \d(\mu_m)\d(\nu_n)^{-1}
        = mn^{-1}.
    \]
    Lastly, we have 
    \(
        \rpa{x}{\d(\mu)}
        = \shiftoff{(\mu_m'\mu_m'')}{(\shiftoff{\mu_m}{x})} 
        = \shiftoff{(\nu_m'\nu_m'')}{(\shiftoff{\nu_n}{y})} 
        = \rpa{y}{\d(\nu)}.
    \)
    Hence, $(x,mn^{-1},y) \in \Fellgpd{\PS{\Lambda}}{\mu}{J}{\nu}{K}$.

    Now fix $(x',\d(\mu)\d(\nu)^{-1},y') \in \Fellgpd{\PS{\Lambda}}{\mu}{J}{\nu}{K}$.
    Recall $\d(\mu)\d(\nu)^{-1} = mn^{-1}$. 
    Because $x'$ is hereditary and $\mu' \preceq \mu$, we have $\mu' \in x'$.
    Similarly, $\nu' \in y'$.
    Thus, $x' \in \Fell{\PS{\Lambda}}{\mu'}{J}$ and $y' \in \Fell{\PS{\Lambda}}{\nu'}{K}$.
    Also, $\mu_m \in x'$ and $\nu_n \in y'$.
    Since $\d(\mu_m) = m$ and $\d(\nu_n) = n$, we have $x' \in \rpadom{m}$ and $y' \in \rpadom{n}$.
    By assumption $\rpa{x'}{\d(\mu)} = \rpa{y'}{\d(\nu)}$, so $\shiftoff{\mu}{x'} = \shiftoff{\nu}{y'}$.
    By definition of $\mu$ and $\nu$, we have 
    \(
        \shiftoff{(\mu_m\mu_m'\mu_m'')}{x'}
        = \shiftoff{(\nu_n\nu_n'\nu_n'')}{y'}.
    \)
    Since $\mu_m'\mu_m'' = \nu_n'\nu_n''$, the above implies 
    \[
        \shiftoff{\mu_m}{x'}
        = \shifton{(\mu_m'\mu_m'')}{(\shiftoff{(\mu_m\mu_m'\mu_m'')}{x'})}
        = \shifton{(\nu_n'\nu_n'')}{(\shiftoff{(\nu_n\nu_n'\nu_n'')}{y'})}
        = \shiftoff{\nu_n}{y'}.
    \]
    Hence, $(x',\d(\mu)\d(\nu)^{-1},y') \in \sdpgcylinder{\Fell{\PS{\Lambda}}{\mu'}{J}}{m}{n}{\Fell{\PS{\Lambda}}{\nu'}{K}}$.
\end{proof}

\section{Comparisons with Spielberg's groupoids}
    \label{sec:reconciliation}

The purpose of this section is to compare our path groupoids from \cref{sec:path_groupoid} with relevant groupoids in the literature. 
The main theorem (\cref{thm:FA_groupoid_isomorphism}) is that, if $\Lambda$ is finitely aligned (i.e. $\FA{\Lambda} = \Lambda$), then our path and boundary-path groupoids are topologically isomorphic to analogous groupoids developed by Spielberg for the finitely aligned setting in \cite{Spi12, Spi14}.
Spielberg has also developed a groupoid approach for nonfinitely aligned $\Lambda$ in \cite{Spi20}. 
The author showed in his thesis \cite{Jon24} that the path groupoid $\PG{\tg}$ of the nonfinitely aligned $2$-graph $\tg$ from \cref{ex:tg} has a different structure of open invariant subsets of its unit space to that of Spielberg's groupoid of $\tg$ (and hence their associated Steinberg algebras \cite{Ste10} and C*-algebras \cite{Ren80} have different ideal structure). 
We omit this result because we intend to give a more comprehensive comparison of our construction and Spielberg's in future work. 
Lastly, we show that Spielberg's groupoid of the finitely aligned relative category of paths $(\FAr{\Lambda}, \Lambda)$ associated to each $P$-graph $\Lambda$ with $\FA{\Lambda} \neq \emptyset$ in \cref{sec:FA_relative_cop} does not generally coincide with our path groupoid $\PG{\Lambda}$.

\subsection{Path space in the finitely aligned case}

We begin by noting important simplifications to the path space when $\FA{\Lambda} = \Lambda$.

\begin{lemma}
    \label{lem:FA_PS_filters}
    \hypothesesFA
    Then, $\PS{\Lambda} = \filters{\Lambda}$ and $\BPS{\Lambda}$ is the closure of $\ultrafilters{\Lambda}$ in $\filters{\Lambda}$.
\end{lemma}
\begin{proof}
    Because $\FA{\Lambda} = \Lambda$, and filters are nonempty, every filter contains some $\lambda \in \FA{\Lambda}$.
    That is, $\PS{\Lambda} = \filters{\Lambda}$.
    Then, $\ultrafilters{\Lambda} \cap \PS{\Lambda} = \ultrafilters{\Lambda}$, so $\BPS{\Lambda}$ is the closure of $\ultrafilters{\Lambda}$ in $\filters{\Lambda}$.
\end{proof}

While discussing the finitely aligned case, we denote the path space by $\filters{\Lambda}$ because of \cref{lem:FA_PS_filters}. 
A crucial difference in the finitely aligned case is that the parameter $K$ for the usual basis for the path space can be assumed to be a subset of $\mu\Lambda$, as follows. 
Among other things, this makes left-shift maps open (as well as being continuous bijections as in \cref{lem:left_shift_cts_bijection}).

\begin{lemma}
    \label{lem:FA_basis_for_topology_on_filters}
    \hypothesesFA
    The collection of 
    \(
        \Fell{\filters{\Lambda}}{\mu}{K},
    \)
    where $\mu \in \Lambda$ and $K \subseteq \mu\Lambda$ is finite, is a basis for $\filters{\Lambda}$.
\end{lemma}
\begin{proof}
    Suppose $x \in \Fell{\filters{\Lambda}}{\mu}{K'}$, where $\mu \in \Lambda$ and $K' \subseteq \Lambda$ is finite.
    For each $\zeta \in K'$, since $\Lambda$ is finitely aligned, there is a finite (possibly empty) $J_\zeta \subseteq \Lambda$ such that 
    \(
        \mu\Lambda \cap \zeta\Lambda 
        = \bigcup_{\gamma \in J_\zeta} \gamma\Lambda.
    \)
    Then, $K \coloneqq \bigcup_{\zeta \in K'} J_\zeta$ is a finite subset of $\mu\Lambda$, and the directed and hereditary properties of filters imply
    \(
        x 
        \in \Fell{\filters{\Lambda}}{\mu}{K}
        \subseteq \Fell{\filters{\Lambda}}{\mu}{K'}.
    \)
\end{proof}

\begin{lemma}
    \label{lem:FA_basis_for_PG}
    \hypothesesFA
    The collection of 
    \(
        \Fellgpd{\filters{\Lambda}}{\kappa}{K}{\lambda}{L},
    \)
    where $\kappa, \lambda \in \Lambda$ and $K \subseteq \kappa\Lambda$ and $L \subseteq \lambda\Lambda$ are finite, is a basis for the topology on $\PG{\Lambda}$.
\end{lemma}
\begin{proof}
    The collection of $\sdpgcylinder{\Fell{\filters{\Lambda}}{\kappa}{K}}{m}{n}{\Fell{\filters{\Lambda}}{\lambda}{L}}$, where $m, n \in P$, $\kappa, \lambda \in \Lambda$ and $K \subseteq \kappa\Lambda$ and $L \subseteq \lambda\Lambda$ are finite, is a basis for the topology on $\PG{\Lambda}$ because of \cref{lem:sdpg_basic_basis} and \cref{lem:FA_basis_for_topology_on_filters}.
    Then, the claim follows from similar arguments to those in the proof of \cref{prop:PS_groupoid_cylinders}.
\end{proof}

\subsection{Spielberg's groupoid in the finitely aligned case}
    \label{sec:Spielberg_groupoid}

We recall Spielberg's groupoid in the finitely aligned case from \cite[Page 729]{Spi12} and \cite[\S 2]{MS22}.
For each $\alpha \in \Lambda$ and $\beta_1, \dots, \beta_n \subseteq \alpha\Lambda \setminus \{\alpha\}$, let $E \coloneqq \alpha\Lambda \setminus \bigcup_{i = 1}^n \beta_i\Lambda$, and let $\hat{E} \coloneqq \setof{x \in \filters{\Lambda}}{E \supseteq x \cap \gamma\Lambda \text{ for some } \gamma \in x}$.
The collection of $\hat{E}$ is a basis of compact sets for a topology $\Spielbergstopology$ on $\filters{\Lambda}$.
(Note that $\filters{\Lambda}$ is denoted by ``$\Lambda^*$'' in \cite{Spi12, MS22}).
Let $\Lambda * \Lambda * \filters{\Lambda}$ be the set of $(\alpha, \beta, x) \in \Lambda \times \Lambda \times \filters{\Lambda}$ such that $\s(\alpha) = \s(\beta) = \r(x)$.
Recall $\shifton{\alpha}{x} \coloneqq \setof{\mu \in \Lambda}{\mu \preceq \alpha\gamma \text{ for some } \gamma \in x}$ for each $x \in \filters{\Lambda}$ with $\r(x) = \s(\alpha)$.
For any $(\alpha, \beta, x), (\alpha', \beta', x') \in \Lambda * \Lambda * \filters{\Lambda}$, let $(\alpha, \beta, x) \sim (\alpha', \beta', x')$ if there are $y \in \filters{\Lambda}$ and $\gamma, \gamma' \in \Lambda$ such that $x = \shifton{\gamma}{y}$, $x' = \shifton{\gamma'}{y}$, $\alpha\gamma = \alpha'\gamma'$, and $\beta\gamma = \beta'\gamma'$.
Then, $\sim$ is an equivalence relation, and we write $\Spielbergsgroupoid{\Lambda}$ for the set of equivalence classes of $\Lambda * \Lambda * \filters{\Lambda}$ with respect to $\sim$.
Let ${\Spielbergsgroupoid{\Lambda}}^{(2)}$ be the pairs $([\alpha, \beta, x], [\gamma, \delta, y]) \in \Spielbergsgroupoid{\Lambda} \times \Spielbergsgroupoid{\Lambda}$ that satisfy $\shifton{\beta}{x} = \shifton{\gamma}{y}$, in which case \cite[Lemma 4.12]{Spi14} implies there are $z \in \filters{\Lambda}$ and $\xi, \eta \in \Lambda$ such that $x = \shifton{\xi}{z}$, $y = \shifton{\eta}{z}$ and $\beta\xi = \gamma\eta$.
Define composition and inversion maps by
\(
    [\alpha, \beta, x][\gamma, \delta, y] \coloneqq [\alpha\xi, \delta\eta, z]
\)
and 
\(
    [\alpha, \beta, x]^{-1} \coloneqq [\beta, \alpha, x],
\)
respectively.
Given a basis $\mathcal{B}$ for the topology $\Spielbergstopology$ on $\filters{\Lambda}$, the collection of $[\alpha, \beta, B] \coloneqq \setof{[\alpha, \beta, x]}{x \in B}$, where $B \in \mathcal{B}$, is a basis under which $\Spielbergsgroupoid{\Lambda}$ is a locally compact Hausdorff \'etale groupoid.

\subsection{Isomorphism with Spielberg's groupoid in the finitely aligned case}
    \label{sec:isom_if_FA}

We show our path groupoid $\PG{\Lambda}$ is topologically isomorphic to $\Spielbergsgroupoid{\Lambda}$ when $\Lambda$ is finitely aligned (\cref{thm:FA_groupoid_isomorphism}).
Recall from \cref{lem:FA_basis_for_topology_on_filters} that the collection of $\Fell{\filters{\Lambda}}{\mu}{K}$, where $\mu \in \Lambda$ and $K \subseteq \mu\Lambda$ is finite, is a basis for $\filters{\Lambda}$ if $\Lambda$ is finitely aligned.

\begin{lemma}
    \label{lem:our_basic_sets_are_Spielberg_basic_sets}
    Let $(Q,P)$ be a weakly quasi-lattice ordered group, and let $\Lambda$ be a (not necessarily finitely aligned) $P$-graph with $\FA{\Lambda} \neq \emptyset$.
    For any $\mu \in \Lambda$ and finite $K \subseteq \mu\Lambda$, $\Fell{\filters{\Lambda}}{\mu}{K} = \hat{E}$ for $E \coloneqq \mu\Lambda \setminus \bigcup_{\nu \in K} \nu\Lambda$.
\end{lemma}
\begin{proof}
    Let $x \in \Fell{\filters{\Lambda}}{\mu}{K}$.
    Suppose for a contradiction that $E \not\supseteq x \cap \mu\Lambda$, so there is $\lambda \in x \cap \mu\Lambda$ such that $\lambda \notin E$.
    Since $E = \mu\Lambda \setminus \bigcup_{\nu \in K} \nu\Lambda$, we must have $\lambda \in \nu\Lambda$ for some $\nu \in K$.
    Since $x$ is hereditary, $\nu \in x$, which contradicts $x \in \Fell{\filters{\Lambda}}{\mu}{K}$.
    Hence, $E \supseteq x \cap \mu\Lambda$, which implies $x \in \hat{E}$.
    Now let $x \in \hat{E}$, so there is $\gamma \in x$ such that $E \supseteq x \cap \gamma\Lambda$.
    In particular, $\gamma \in E \subseteq \mu\Lambda$, so $\mu \in x$.
    Suppose for a contradiction that there is $\nu \in K \cap x$.
    Since $x$ is directed, there is $\lambda \in x$ such that $\gamma, \nu \preceq \lambda$.
    Thus, $E \supseteq x \cap \gamma\Lambda \ni \lambda$, so $\lambda \in E$. 
    However, $\lambda \in \nu\Lambda$, contradicting the definition of $E$.
    Hence, $x \in \Fell{\filters{\Lambda}}{\mu}{K}$.
\end{proof}
    
The only difference to the parameterization of the $\hat{E}$ in \cref{sec:Spielberg_groupoid} when defining $\Spielbergstopology$ and the parameterization of the $\hat{E}$ appearing in \cref{lem:our_basic_sets_are_Spielberg_basic_sets} is that our $K$ could contain $\mu$, in which case $\Fell{\filters{\Lambda}}{\mu}{K} = \emptyset$. 
That is, the basis for our topology on $\filters{\Lambda}$ from \cref{lem:FA_basis_for_topology_on_filters} differs from the basis for the topology $\Spielbergstopology$ on $\filters{\Lambda}$ only by $\emptyset \subseteq \filters{\Lambda}$. 
Therefore:

\begin{lemma}
    \label{lem:FA_recon_unit_spaces_coincide}
    \hypothesesFA
    Then, $\Spielbergstopology$ equals the topology on $\filters{\Lambda}$ given by the basis from \cref{lem:FA_basis_for_topology_on_filters}.
\end{lemma}

\begin{theorem}
    \label{thm:FA_groupoid_isomorphism}
    \hypothesesFA
    The map 
    \[
        \Spielbergsgroupoid{\Lambda} \to \PG{\Lambda},
        \quad 
        [\alpha, \beta, x] \mapsto (\shifton{\alpha}{x}, \d(\alpha)\d(\beta)^{-1}, \shifton{\beta}{x})
    \]
    is a topological groupoid isomorphism. 
    Moreover, the reductions $\Spielbergsgroupoid{\Lambda}|_{\partial\Lambda}$ and $\BPG{\Lambda}$ are topologically isomorphic.
\end{theorem}

\begin{proof}
    Our first aim is to show the map is a well-defined bijection.
    Let $[\alpha, \beta, x] \in \Spielbergsgroupoid{\Lambda}$.
    Since $\rpa{\shifton{\alpha}{x}}{\d(\alpha)} = x = \rpa{\shifton{\beta}{x}}{\d(\beta)}$, $(\shifton{\alpha}{x}, \d(\alpha)\d(\beta)^{-1}, \shifton{\beta}{x}) \in \PG{\Lambda}$.
    If $[\alpha, \beta, x] = [\alpha', \beta', x']$, then there are $y \in \filters{\Lambda}$ and $\gamma, \gamma' \in \Lambda$ such that $x = \shifton{\gamma}{y}$, $x' = \shifton{\gamma'}{y}$, $\alpha\gamma = \alpha'\gamma'$, and $\beta\gamma = \beta'\gamma'$.
    Then,
    \(
        \shifton{\alpha}{x} 
        = \shifton{\alpha}{(\shifton{\gamma}{y})}
        = \shifton{(\alpha\gamma)}{y}
        = \shifton{(\alpha'\gamma')}{y}
        = \shifton{\alpha'}{(\shifton{\gamma'}{y})}
        = \shifton{\alpha'}{x'}.
    \)
    A similar argument shows $\shifton{\beta}{x} = \shifton{\beta'}{x'}$.
    Lastly, $\d(\alpha)\d(\beta)^{-1} = \d(\alpha')\d(\beta')^{-1}$ follows from the identities $\alpha \gamma = \alpha' \gamma'$ and $\beta \gamma = \beta' \gamma'$ and the fact that $\d$ is a functor.
    Hence, the map is well-defined.
    The map is surjective because any $(x, q, y) \in \PG{\Lambda}$ is of the form $(\shifton{\mu}{x}, \d(\mu)\d(\nu)^{-1}, \shifton{\nu}{x})$ for some $x \in \filters{\Lambda}$ by \cref{prop:PS_groupoid_characterisation}, which is the image of $[\mu, \nu, x]$.
    For injectivity, suppose $[\alpha, \beta, x], [\alpha', \beta', x'] \in \Spielbergsgroupoid{\Lambda}$ are such that $(\shifton{\alpha}{x}, \d(\alpha)\d(\beta)^{-1}, \shifton{\beta}{x}) = (\shifton{\alpha'}{x'}, \d(\alpha')\d(\beta')^{-1}, \shifton{\beta'}{x'})$.
    Notice $\alpha, \alpha' \in \shifton{\alpha}{x} = \shifton{\alpha'}{x'}$, which is directed, so there are $\gamma, \gamma' \in \Lambda$ such that $\alpha\gamma = \alpha'\gamma' \in \shifton{\alpha}{x} = \shifton{\alpha'}{x'}$.
    Notice $\gamma \in x$, so $y \coloneqq \shiftoff{\gamma}{x} \in \filters{\Lambda}$.
    Then, $x = \shifton{\gamma}{(\shiftoff{\gamma}{x})} = \shifton{\gamma}{y}$, and $x' = \shifton{\gamma'}{(\shiftoff{\gamma'}{x'})} = \shifton{\gamma'}{y}$.
    To show $\beta\gamma = \beta'\gamma'$, we show $\beta\gamma, \beta'\gamma' \in \shifton{\beta}{x}$ and $\d(\beta\gamma) = \d(\beta'\gamma')$, and we apply the injectivity of $\d|_{\shifton{\beta}{x}}$.
    Since $\gamma \in x$, we have $\beta\gamma \in \shifton{\beta}{x}$.
    Also, $\gamma' \in x'$, so $\beta'\gamma' \in \shifton{\beta'}{x'} = \shifton{\beta}{x}$.
    The identity $\alpha\gamma = \alpha'\gamma'$ implies $\d(\alpha')^{-1}\d(\alpha) = \d(\gamma')\d(\gamma)^{-1}$.
    If $\d(\gamma')\d(\gamma)^{-1} = \d(\beta')^{-1}\d(\beta)$, then $\d(\beta\gamma) = \d(\beta'\gamma')$, so it remains to show $\d(\alpha')^{-1}\d(\alpha) = \d(\beta')^{-1}\d(\beta)$, which follows from the assumption that $\d(\alpha)\d(\beta)^{-1} = \d(\alpha')\d(\beta')^{-1}$.
    Therefore, $\beta\gamma = \beta'\gamma'$ by the injectivity of $\d|_{\shifton{\beta}{x}}$.
    Thus, $[\alpha, \beta, x] = [\alpha', \beta', x']$, so the map is a bijection.

    Next we aim to show the map is a homeomorphism.
    The image of an arbitrary basic open set $[\alpha, \beta, \Fell{\filters{\Lambda}}{\mu}{K}]$, where $\alpha, \beta, \mu \in \Lambda$ and $K \subseteq \mu\Lambda$ is finite, is the set of $(\shifton{\alpha}{x}, \d(\alpha)\d(\beta)^{-1}, \shifton{\beta}{x})$ such that $x \in \Fell{\filters{\Lambda}}{\mu}{K}$.
    We claim this image equals 
    \[
        I \coloneqq \Fellgpd{\filters{\Lambda}}{\alpha\mu}{\setof{\alpha\nu}{\nu \in K}}{\beta\mu}{\setof{\beta\nu}{\nu \in K}}.
    \]
    Consider $(\shifton{\alpha}{x}, \d(\alpha)\d(\beta)^{-1}, \shifton{\beta}{x})$ for some $x \in \Fell{\filters{\Lambda}}{\mu}{K}$.
    Observe $\alpha\mu \in \shifton{\alpha}{x}$.
    If there is a $\nu \in K$ such that $\alpha\nu \in \shifton{\alpha}{x}$, then $\nu \in x \cap K$, which is a contradiction.
    A similar argument applies to $\shifton{\beta}{x}$, so we have $\shifton{\alpha}{x} \in \Fell{\filters{\Lambda}}{\alpha\mu}{\setof{\alpha\nu}{\nu \in K}}$ and 
    $\shifton{\beta}{x} \in \Fell{\filters{\Lambda}}{\beta\mu}{\setof{\beta\nu}{\nu \in K}}$.
    Lastly, notice $\d(\alpha)\d(\beta)^{-1} = \d(\alpha\mu)\d(\beta\mu)^{-1}$, and
    \[
        \rpa{(\shifton{\alpha}{x})}{\d(\alpha\mu)} 
        = \shiftoff{(\alpha\mu)}{(\shifton{\alpha}{x})} 
        = \shiftoff{\mu}{x} 
        = \shiftoff{(\beta\mu)}{(\shifton{\beta}{x})} 
        = \rpa{(\shifton{\beta}{x})}{\d(\beta\mu)}.
    \]
    Hence, $(\shifton{\alpha}{x}, \d(\alpha)\d(\beta)^{-1}, \shifton{\beta}{x}) \in I$.
    Now fix $(x, \d(\alpha\mu)\d(\beta\mu)^{-1}, y) \in I$, so $\shiftoff{(\alpha\mu)}{x} = \shiftoff{(\beta\mu)}{y}$, which implies $z \coloneqq \shiftoff{\alpha}{x} = \shiftoff{\beta}{y}$ and $(x, \d(\alpha\mu)\d(\beta\mu)^{-1}, y) = (\shifton{\alpha}{z}, \d(\alpha)\d(\beta)^{-1}, \shifton{\beta}{z})$.
    Since $\alpha\mu \in x$, $\mu \in \shiftoff{\alpha}{x} = z$.
    There cannot be $\nu \in K \cap z$, otherwise $\alpha\nu \in x$.
    Hence, $z \in \Fell{\filters{\Lambda}}{\mu}{K}$, and so $[\alpha, \beta, z] \in [\alpha, \beta, \Fell{\filters{\Lambda}}{\mu}{K}]$.
    That is, the image of the basic open set $[\alpha, \beta, \Fell{\filters{\Lambda}}{\mu}{K}]$ equals $I$.
    An arbitrary basic open set $\Fellgpd{\filters{\Lambda}}{\kappa}{K}{\lambda}{L}$ in $\PG{\Lambda}$ (see \cref{lem:FA_basis_for_PG}) is in the form of $I$ with $\alpha = \kappa$, $\beta = \lambda$, $\mu = \s(\alpha) = \s(\beta)$ and $J = \shiftoff{\kappa}{K} \cup \shiftoff{\lambda}{L}$.
    Hence, basic open sets in $\PG{\Lambda}$ identify with basic open sets in $\Spielbergsgroupoid{\Lambda}$ via the map, so the map is a homeomorphism.

    We now aim to show the groupoids are isomorphic.
    If $([\alpha, \beta, x], [\gamma, \delta, y]) \in {\Spielbergsgroupoid{\Lambda}}^{(2)}$, then $\shifton{\beta}{x} = \shifton{\gamma}{y}$, which implies $((\shifton{\alpha}{x}, \d(\alpha)\d(\beta)^{-1}, \shifton{\beta}{x}), (\shifton{\gamma}{y}, \d(\gamma)\d(\delta)^{-1}, \shifton{\delta}{y})) \in \composablepairs{\PG{\Lambda}}$.
    The converse holds in the same way, so the map identifies $\composablepairs{\PG{\Lambda}}$ with $\composablepairs{{\Spielbergsgroupoid{\Lambda}}}$.
    Recall that $([\alpha, \beta, x], [\gamma, \delta, y]) \in {\Spielbergsgroupoid{\Lambda}}^{(2)}$ implies there are $z \in \filters{\Lambda}$ and $\xi, \eta \in \Lambda$ such that $x = \shifton{\xi}{z}$, $y = \shifton{\eta}{z}$ and $\beta\xi = \gamma\eta$. 
    Then, 
    \[
        [\alpha, \beta, x][\gamma, \delta, y] 
        = [\alpha\xi, \delta\eta, z]
        \mapsto (\shifton{(\alpha\xi)}{z}, \d(\alpha\xi)\d(\delta\eta)^{-1}, \shifton{(\delta\eta)}{z}).
    \]
    Also,
    \[
        (\shifton{\alpha}{x}, \d(\alpha)\d(\beta)^{-1}, \shifton{\beta}{x})(\shifton{\gamma}{y}, \d(\gamma)\d(\delta)^{-1}, \shifton{\delta}{y}) 
        = (\shifton{\alpha}{x}, \d(\alpha)\d(\beta)^{-1}\d(\gamma)\d(\delta)^{-1}, \shifton{\delta}{y}).
    \]
    We have $\shifton{(\alpha\xi)}{z} = \shifton{\alpha}{x}$ and $\shifton{(\delta\eta)}{z} = \shifton{\delta}{y}$.
    Lastly, 
    \begin{align*}        
        \d(\alpha)\d(\beta)^{-1}\d(\gamma)\d(\delta)^{-1} 
        &= \d(\alpha)\d(\xi)\d(\xi)^{-1}\d(\beta)^{-1}\d(\gamma)\d(\eta)\d(\eta)^{-1}\d(\delta)^{-1} \\
        &= \d(\alpha\xi)\d(\beta\xi)^{-1}\d(\gamma\eta)\d(\delta\eta)^{-1} \\
        &= \d(\alpha\xi)\d(\beta\xi)^{-1}\d(\beta\xi)\d(\delta\eta)^{-1} \\
        &= \d(\alpha\xi)\d(\delta\eta)^{-1}.
    \end{align*}
    Thus, the map preserves composition.
    It is immediate from the definitions of the inversions and the definition of the map that inversion is preserved too.
    Therefore, the map is a topological groupoid isomorphism.
    Our boundary-path space $\BPS{\Lambda}$ equals the boundary of $\Lambda$ defined by Spielberg on \cite[Page 729]{Spi12}, there too denoted by ``$\partial\Lambda$''.
    Our boundary-path groupoid $\BPG{\Lambda}$ is defined to be the reduction of $\PG{\Lambda}$ to $\BPS{\Lambda}$, which therefore coincides with the reduction $\Spielbergsgroupoid{\Lambda}|_{\partial\Lambda}$ of $\Spielbergsgroupoid{\Lambda}$ to $\BPS{\Lambda}$.
\end{proof}

\begin{remark}
    The author showed in \cite{Jon24} that our path and boundary-path groupoids coincide also with those of Renault and Williams \cite{RW17} and Yeend \cite{Yee07} in the discrete row-finite case.
\end{remark}

\subsection{Inverse semigroup model in the finitely aligned case}

In \cite[Definition 2.2]{OP20}, Ortega and Pardo associate an inverse semigroup $\mathcal{S}_\Lambda$ to any left cancellative small category $\Lambda$ (in particular, any $P$-graph). 
Exel's tight groupoid of $\mathcal{S}_\Lambda$ from \cite{Exe08} is denoted by $\mathcal{G}_{\mathrm{tight}}(\mathcal{S}_\Lambda)$.
When $\Lambda$ is finitely aligned, $\Spielbergsgroupoid{\Lambda}|_{\partial\Lambda}$ is topologically isomorphic to $\mathcal{G}_{\mathrm{tight}}(\mathcal{S}_\Lambda)$ by \cite[Lemma 4.9 and Proposition 5.2]{OP20}, so we have the following corollary to \cref{thm:FA_groupoid_isomorphism}.

\newpage
\begin{corollary}
    \label{cor:inverse_semigroup_model}
    \hypothesesFA
    The following groupoids are topologically isomorphic:
    \begin{itemize}
        \item $\BPG{\Lambda}$ from \cref{def:boundary_path_groupoid},
        \item $\Spielbergsgroupoid{\Lambda}|_{\partial\Lambda}$ from \cite{Spi20}, and 
        \item $\mathcal{G}_{\mathrm{tight}}(\mathcal{S}_\Lambda)$ from \cite{OP20}.
    \end{itemize}
\end{corollary}

\begin{remark}
    Another relevant inverse semigroup model for groupoids associated to higher-rank graphs is the inverse semigroup of \cite[Definition 4.1]{FMY05}, which we denote here by $\mathcal{S}_\Lambda^\mathrm{FMY}$.
    The author is unaware of a comparison between $\mathcal{S}_\Lambda$ from \cite{OP20} and $\mathcal{S}_\Lambda^\mathrm{FMY}$ beyond the observation in \cite{LSV24} that certain elements of $\mathcal{S}_\Lambda$ are ``closely related to the building blocks of'' $\mathcal{S}_\Lambda^\mathrm{FMY}$.
    If $\mathcal{S}_\Lambda$ is isomorphic to $\mathcal{S}_\Lambda^\mathrm{FMY}$, then our boundary-path groupoid coincides also with that of \cite{FMY05}.
\end{remark}

\subsection{Spielberg's groupoid of a relative category of paths in general}
    \label{sec:gpd_of_relative_cop}

Recall from \cref{sec:FA_relative_cop} that $(\FAr{\Lambda}, \Lambda)$ is a \textit{finitely aligned} relative category of paths for any $P$-graph $\Lambda$ with $\FA{\Lambda} \neq \emptyset$. 
Spielberg associates a groupoid $G(\Lambda', \Lambda)$ to each relative category of paths $(\Lambda', \Lambda)$ in \cite[Definition 14.5]{Spi20}.
In this section, we show $G(\FAr{\Lambda}, \Lambda)$ is not topologically isomorphic to our path groupoid $\PG{\Lambda}$ in general.
We first state \cite[Remark 3.7]{Spi14} explicitly:

\begin{proposition}[Remark 3.7 of \cite{Spi14}]
    If $(\Lambda', \Lambda)$ is a finitely aligned relative category of paths, then $G(\Lambda', \Lambda)$ is topologically isomorphic to the groupoid $G_{\Lambda'}^\mathrm{Spi}$ from \cref{sec:Spielberg_groupoid} of the finitely aligned category of paths $\Lambda'$.
\end{proposition}

In particular, $G(\FAr{\Lambda}, \Lambda)$ is topologically isomorphic to $G_{\FAr{\Lambda}}^\mathrm{Spi}$. 
By \cref{lem:FA_recon_unit_spaces_coincide}, the unit space of $G(\FAr{\Lambda}, \Lambda)$ is homeomorphic to $\filters{\FAr{\Lambda}}$, which has the basis from \cref{lem:FA_basis_for_topology_on_filters}. 
Recall the $2$-graph $\tg$ from \cref{ex:tg}, for which $\FAr{\tg} = \FA{\tg} \cup \{v\} = \tg \setminus \{\lambda, \mu\}$. 
The nondiscrete elements of $\filters{\FAr{\tg}}$ are $\principal{t} = \lim_{n \to \infty} \principal{\beta_n}$, $\principal{w} = \lim_{n \to \infty} \principal{\alpha_n}$ and $\principal{v} = \lim_{n \to \infty} \principal{\lambda\alpha_n}$. 
This third nondiscrete element is due to the fact that $\FAr{\tg}$ does not contain $\lambda$ or $\mu$, so the basic open sets containing $\principal{v}$ are of the form 
\[
    \Fell{\filters{\FAr{\tg}}}{v}{\{\lambda\alpha_{n_1}, \dots, \lambda\alpha_{n_k}\}}.
\]
On the other hand, the only nondiscrete elements of $\PS{\tg}$ are $\principal{t} = \lim_{n \to \infty} \principal{\beta_n}$ and $\principal{w} = \lim_{n \to \infty} \principal{\alpha_n}$. 
Therefore, $\filters{\FAr{\tg}}$ is not homeomorphic to our path space $\PS{\tg}$, and so $G(\FAr{\tg}, \tg)$ is not topologically isomorphic to our path groupoid $\PG{\tg}$.

\section*{Acknowledgements}

The theorems of \cref{sec:path_space,sec:path_groupoid} were first proved in the author's thesis \cite{Jon24} under the supervision of Lisa Orloff Clark and Astrid an Huef. 
Their support and ideas have been pivotal. 
The author is also grateful to thesis reviewers whose ideas led to the results in \cref{sec:finite_alignment,sec:reconciliation}. 
This research is supported by the Marsden Fund grant 21-VUW-156 from the Royal Society of New Zealand and grant P1-0288 from the Slovenian Research and Innovation Agency.

\printbibliography

\end{document}

\typeout{get arXiv to do 4 passes: Label(s) may have changed. Rerun}